\documentclass[12pt]{article}
\RequirePackage[OT1]{fontenc}
\RequirePackage{amsthm,amsmath}
\RequirePackage[numbers]{natbib}
\RequirePackage[colorlinks,citecolor=blue,urlcolor=blue]{hyperref}

\usepackage{eurosym}
\usepackage{amssymb}
\usepackage{amsfonts}
\usepackage{enumerate}
\usepackage[normalem]{ulem}
\usepackage{color}
\usepackage[dvipsnames]{xcolor}
\usepackage[margin=1.2in]{geometry}

\newtheorem{theorem}{Theorem}[section]
\newtheorem{corollary}[theorem]{Corollary}
\newtheorem{lemma}[theorem]{Lemma}
\newtheorem{proposition}[theorem]{Proposition}

\theoremstyle{remark}
\newtheorem{definition}[theorem]{Definition}
\theoremstyle{remark}
\newtheorem{remark}[theorem]{Remark}

\newcommand{\E}{\mathbb{E}}
\newcommand{\N}{\mathbb{N}}
\newcommand{\Z}{\mathbb{Z}}
\newcommand{\R}{\mathbb{R}}
\renewcommand{\P}{\mathbb{P}}

\newcommand{\wt}{\widetilde}
\newcommand{\eps}{\epsilon}

\newcommand{\xin}[1]{{\color{blue} #1}}

\newcommand{\greg}[1]{{\color{purple} #1}}
\newcommand{\rev}[1]{{\color{blue} #1}}
\newcommand{\old}[1]{{}}

\def \F {{\cal F}}

\newcommand{{\pe}}  {\partial_e}
\newcommand {{\lodd}} {{\mathcal J}}
\newcommand{{\inrad}} {{\rm inrad}}

\newcommand {{\cent}} {{w_0}}
\newcommand {{\eb}}  {{\bf e}}
\newcommand{{\niceset}}  {{\mathcal X}}
\newcommand {{\nicesettwo}} {\tilde{\niceset}}
\newcommand {{\sprob}}{\mathcal P}
\newcommand {{\smallbit}}{{\mathcal S}}
\newcommand {{\Prob}}{{\mathbb P}}
\newcommand {{\eset}}  {{\emptyset}}
\newcommand {{\Es}}  {{\rm Es}}
\newcommand {{\Cp}} {{\rm cap}}
\newcommand {{\p}}  {{\partial}}
\newcommand {{\hm}}  {{\rm hm}}

\newcommand {{\LE}}  {{\mathbf{LE}}}

\newcommand {{\saws}}  {{\mathcal K}}

\begin{document}

\author{%
\begin{tabular}{c}
Gregory F. Lawler\\
[-4pt] 
{\small University of Chicago}%
\end{tabular}
\and 
\begin{tabular}{c}
Xin Sun \\[-4pt] 
{\small Columbia University}%
\end{tabular}
\and 
\begin{tabular}{c}
Wei Wu
\\[-4pt] 
{\small University of Warwick}%
\end{tabular}%
}
\title{Four dimensional loop-erased random walk}
\date{}
\maketitle

\begin{abstract}
	The loop-erased
	random walk (LERW) in $\mathbb{Z}^4$ is the process
	obtained by erasing loops chronologically for simple random walk.  We  prove that the escape probability of the LERW renormalized by $(\log n)^{\frac{1}{3}}$ converges almost
	surely and in $L^{p}$ for all $p>0$.
	Along the way, we extend  previous results by the first author building on slowly recurrent sets. We provide two applications for the escape probability. We construct the
	two-sided LERW, and  we construct a $\pm 1$ spin model coupled with the wired spanning
	forests on $\Z^4$
	with the bi-Laplacian Gaussian field on $\mathbb{R}^{4}$ as its scaling limit.
\end{abstract}

\section{Introduction}\label{secintro}
Loop-erased random walk (LERW) is a probability measure on self-avoiding paths introduced by the first author of this paper in \cite{Lawler-Duke}. Since then LERW has become an important model in  statistical physics and probability, with close connections to other important subjects such as uniform spanning tree and Schramm-Loewner evolution. 
A key quantity that governs the large scale behavior of LERW is the so-called escape probability, namely, the non-intersection probability of a LERW and an independent simple random walk (SRW) starting at the same point. It is known that $d=4$ is critical for LERW, in the sense that a  LERW and an SRW on $\mathbb{Z}^d$ intersect a.s. if and only if $d \leq 4$. It iwas shown in
\cite{L4Duke} that LERW on $\Z^4$ has Brownian motion as its scaling limit after proper normalization.  The exact normalization was conjectured but not proved in that paper; in  \cite{3rdmoment} 
it was determind  up to multiplicative constants. The 
argument uses a weak version of a ``mean-field''
property for LERW in $\mathbb{Z}^4$.
In this paper, we establish the sharp mean-field property for the escape probability of LERW on $\mathbb{Z}^4$ that goes beyond the scaling limit result. 

We state our main results for the renormalized escape probability of 4D LERW, Theorem \ref{thm:inter} and \ref{thm:const},  in Section \ref{subsec:nis}. An outline of the proofs is given in Section \ref{subsec:outline}. Then in Section \ref{subsec:two-sided} and \ref{subsec:field} we discuss two applications of the main results, namely a construction of the two-sided LERW in $d=4$, and a spin field  coupled with the wired spanning
forests on $\Z^4$ with the 
bi-Laplacian Gaussian field on $\mathbb{R}^{4}$ as its scaling limit.

\subsection{Escape probability of LERW}\label{subsec:nis}
Given a positive integer $d$,  a process $S=\{S_n\}_{n\in \N}$ on $\Z^d$ is called  a \emph{simple random walk} (SRW) on $\Z^d$ if $\{S_{n+1}-S_n \}_{n\in\N}$ are i.i.d. random variables taking uniform distribution on $\{z\in \Z^d: |z|=1\}$. Here $|\cdot |$ is the Euclidean norm on $\R^d$. Unless otherwise stated, our  SRW starts at the origin, namely, $S_0=0$. When $S_0=x$ almost surely, we denote the probability measure of $S$ by $\P^x$. 

A \emph{path} on $\Z^d$ is a sequence of vertices such that any two consecutive vertices are neighbors in $\Z^d$.  Given a sample $S$ of SRW and $m<n\in \N$, let $S[m,n]$ and  $S[n,\infty)$  be the paths $[S_m,S_{m+1}\cdots, S_n]$ and $[S_n,S_{n+1},\cdots]$ respectively. 
Given a finite path $\mathcal{P}=[v_0,v_1, \dots,v_n]$ on $\Z^d$,
the \emph{(forward) loop erasure} of $\mathcal{P}$ (denoted by $\LE(\mathcal{P})$) is defined
by erasing cycles in $\mathcal{P} $ chronologically. More precisely,  we define $\LE(\mathcal{P})$ inductively as follows. The first vertex $u_0$ of $\LE(\mathcal{P})$ is
the vertex $v_0$ of $\mathcal{P}$. Supposing that $u_j$ has been set, let $k$ be the last index such that $v_k=u_j$. Set $u_{j+1}=v_{k+1}$ if $k<n$; otherwise, let
$\LE(\mathcal{P}) :=[u_0,\dots, u_j]$.  Suppose $S$ is an  SRW on $\Z^d$ ($d\ge 3$).  Since $S$ is transient, there is no trouble defining $\LE(S)=\LE(S[0,\infty))$, which we call the \emph{loop-erased random walk} (LERW) on $\Z^d$. LERW on $\Z^2$ can be defined via a limiting procedure but we will not discuss it in this paper.

Let $W$ and $S$ be two
independent simple random walks on $\mathbb{Z}^4$ starting at the origin
and $\eta=\LE(S)$.  Let 
\begin{equation*}
	X_n=(\log n)^{\frac13}\P \{W[1,n^2]\cap \eta=\emptyset|\eta\}.
\end{equation*}%
In \cite{3rdmoment}, building on the work on slowly recurrent sets 
\cite{LSlowly}, the first author of this paper proved that $\mathbb{E}[X_{n}^{p}]\asymp 1$
for all $p>0$. In this paper, we show that

\begin{theorem}
	\label{thm:inter} 
	There exists 	a nontrivial random variable $X_\infty$ such that 
	\[
	\lim_{n\to\infty} X_n=X_\infty\qquad \textrm{almost surely and in $L^p$ for all $p>0$}.
	\]
\end{theorem}

We can view $X_{\infty }$ as the renormalized escape probability of 4D
LERW at its starting point. It is the key for our construction of the 4D two-sided LERW in Section~\ref{subsec:two-sided}.
Our next theorem is similar to Theorem~\ref{thm:inter} with the additional feature of the evaluation of the limiting  constant. 
\begin{theorem}
	\label{thm:const} Let $W,W',W'',S$ be four
	independent simple random walks on $\mathbb{Z}^4$ starting from the origin
	and $\eta=\LE(S)$.   Then
	\begin{align*}  \label{eq:exact}
		&\lim_{n\to\infty}(\log n)\P \{(W[1,n^2]\cup W^{\prime }[1,n^2])\cap
		\eta=\emptyset, W^{\prime \prime }[0,n^2]\cap\eta=\{0\} \}  
		= \frac{\pi^2}{24}.
	\end{align*}
\end{theorem}
Write $\frac{\pi^2}{24}$ in Theorem~\ref{thm:const} as $\frac{1}{3}\cdot \frac{\pi^2}{8}$.
We will see that  the constant $\frac{1}{3}$ is universal  and  is the reciprocal of the number of SRW's other than $S$. 
The factor $\pi ^{2}/8$ comes from the bi-harmonic Green function of $\mathbb{Z}^{4}$ evaluated at $
(0,0)$ and is  lattice-dependent. 
{The SRW analog of Theorem~\ref{thm:const} is proved in \cite[Corollary 4.2.5]{lawler-intersection}:
	\begin{align*}
		\lim_{n\rightarrow \infty }(\log n)\P \{W[1,n^2]\cap S[0,n^2]=
		\emptyset,W^{\prime }[0,n^2]\cap S[1,n^2]=\emptyset \}=\frac{1}{2}\cdot \frac{\pi ^{2}}{8}.
	\end{align*}}
	
	Theorem~\ref{thm:inter} and ~\ref{thm:const} are  special case of our Theorem~\ref{maintheoremjune}, whose proof is outlined in Section~\ref{subsec:outline}.
	In particular, the asymptotic result  is obtained from a refined analysis  of slowly recurrent set beyond  \cite{LSlowly,3rdmoment} 
	as well as fine estimates on the harmonic measure of 4D LERW.
	The explicit constant $\frac{\pi^2}{24}$ is  obtained from a ``first passage''
	path decomposition of the intersection of an SRW and a LERW. Here care is needed because there are
	several time scales involved.  See Section~\ref{subsec:outline} for an outline.
	As a byproduct, at the end of Section~\ref{sec:long} we  obtain an asymptotic result on the long range intersection between SRW and LERW  which is of independent interest.
	
	{To state the result, we recall the Green function on $\Z^4$ defined by
		\[
		G(x,y) = \sum_{n=0}^{\infty}\P^x[S_n=y],
		\] Given a subset  $A\subset\Z^4$, the Green function on $A$ is defined by
		\[
		G_A(x,y) = \sum_{n=0}^{\infty}\P^x[S_n=y,S[0,n]\subset A].
		\]  It will be technically easier to work on geometric scales. Let  $C_n=\{z\in \Z^4: |z|< e^n \}$ be the discrete
		disk,   $G_n = G_{C_n}$, and 
		\begin{equation*}
			G_n^2(w) = \sum_{z \in C_n} G_n(0,z) \, G_n(z,w).
		\end{equation*}
		\begin{theorem}\label{thm:long-range}
			Let $W,S$ be independent simple random walks on $\Z^4$ with $W_0=0$ and $S_0=w$. Let \(	\sigma^W_n = \min\{j: W_j \not\in C_n\}\) and \(	\sigma_n = \min\{j: S_j \not\in C_n\}.\) If
			\[     q_n (w)= \Prob\{W[0, \sigma^W_n] \cap \LE(S[0,\sigma_n]) \neq \eset\} ,\]
			then 
			\begin{equation*}
				\lim_{n \rightarrow \infty}\; \max_{n^{-1} \leq e^{-n} \, |w| \leq 1 -
					n^{-1} } \left| n\,q_n(w) -
				\frac{\pi^2}{24} \, G^2_n(w) \right| = 0 .
			\end{equation*}
		\end{theorem}

		\begin{remark}
			Theorem~\ref{thm:long-range} holds if 	$W[0, \sigma^W_n] \cap \LE(S[0,\sigma_n])$ is replaced by $W[0, \sigma^W_n] \cap S[0,\sigma_n]$ and $\pi^2/24$ is replaced by $\pi^2/16$. This is the long range estimate for two independent SRW's in \cite[Section 4.3]{lawler-intersection}. The function $G_n^2(w)$ is the expected number of intersections of $S[0,\sigma_n]$ and $W[0,\sigma^W_n]$. This means that the long-range non-intersection probability of  an SRW and an independent LERW is comparable with that of two independent SRW's. This is closely related to  the fact that the scaling limit of LERW  on $\Z^4$ is Brownian motion, that
			is, has Gaussian limits.
		\end{remark}
	}
	
	\subsection{Outline of the proof}\label{subsec:outline}
	
	Here we state and give an outline for the proof
	of  Theorem~\ref{maintheoremjune} from
	which  Theorem~\ref{thm:inter} and~\ref{thm:const} 
	are immediate corollaries.  The details
	for the proof are given  in Section~\ref{sec:pre}-\ref{constantsec}.
	
	We start by defining some notation. 
	Let \(	\sigma_n = \min\left\{j \geq 0: S_j \not\in C_n \right\}\) and 
	\begin{equation}\label{eq:def-Fn}
		\textrm{$\F_n$ be the $\sigma$-algebra generated by $\{S_j: j \leq\sigma_n\}$.}
	\end{equation}
	We recall that there exist $0 <\beta,c <\infty$ such that for all
	$n$, if $z \in C_{n-1}$ and $a \geq 1$,
	\begin{equation}\label{jul18.1}
		\Prob^z\left\{ a^{-1} \, e^{2n}
		\leq \sigma_n \leq a \, e^{2n} \right\} \geq 1 - c
		\, e^{-\beta a}.
	\end{equation}
	For the lower inequality, see, e.g.. \cite[(12.12)]{lawler-limic},
	and the upper inequality follows from the
	fact that
	$ \Prob^z\{\sigma_n  \leq (k+1) e^{2n} 
	\mid \sigma_n \geq k \, e^{2n}\}  $ is uniformly
	bounded away from $0$.
	
	If $x \in \Z^4,V \subset \mathbb{Z}^4$, we write 
	\begin{equation*}
		H(x,V) = \Prob^x\{S[0,\infty) \cap V \neq \eset \}, \] 
		\[ H(V) = H(0,V),\;\;\;\;\Es(V) = 1 - H(V) ,
		\end{equation*}
		\[
		\overline H(x,V) = \Prob^x\{S[1,\infty) \cap V \neq \eset \}, \;\;\;\;
		\overline{\Es}(V) = 1 - \overline H(0,V).
		\]
		Note that $\overline {\Es}(V) = \Es(V)$, if $0 \not \in V$. If $0 \in V$,    a
		standard last-exit decomposition shows that 
		\begin{equation}\label{eq:last-exit}
			\Es(V^0) = G_{\mathbb{Z}^4\setminus V^0}(0,0) \,\overline{\Es}(V),
		\end{equation}
		where $V^0 = V \setminus \{0\}$ and $G_{\mathbb{Z}^4\setminus V^0}$ is the
		Green's function on  $\mathbb{Z}^4\setminus V^0$.
		We also write
		\[   \Es(V;n) = \Prob\{S{[1,\sigma_n]} \cap V
		= \eset \}, \]
		which is clearly decreasing in $n$.
		
		We have to be a little careful about the definition of the loop-erasures of
		the random walk and loop-erasures of subpaths of the walk. We will use the
		following notations.
		
		\begin{itemize}
			\item $\eta$ denotes the (forward) loop-erasure of $S[0,\infty)$
			and 
			\begin{equation*}
				\Gamma =\eta[1,\infty) = \eta[0,\infty) \setminus\{0\}.
			\end{equation*}
			
			\item $\omega_n$ denotes the finite random walk path $S[\sigma_{n-1},
			\sigma_n]$
			
			\item $\eta^n = \LE(\omega_n)$ denotes the loop-erasure of $%
			S[\sigma_{n-1},\sigma_n]$.
			
			\item $\Gamma_n$ = $\LE(S[0,\sigma_n]) \setminus \{0\}$, that is, $\Gamma_n$
			is the loop-erasure of $S[0,\sigma_n]$ with the origin removed.
		\end{itemize}
		
		Note that $S[1,\infty)$ is the concatenation of the paths $%
		\omega_1,\omega_2,\ldots$. However, it is not true that $\Gamma $ is the
		concatenation of $\eta^1, \eta^2,\ldots$, and that is one of the technical
		issues that must be addressed in the proof.
		
		Let $Y_{n},Z_{n}, G_n$ be the $\F_{n}$-measurable random variables 
		\begin{equation}\label{eq:def-Zn}
			Y_{n}=H(\eta ^{n}),\;\;\;\;\;Z_{n}=\Es[\Gamma_n],\;\;\;\;G_{n}=G_{\mathbb{Z}%
				^{4}\setminus \Gamma _{n}}(0,0).
		\end{equation}%
		By \eqref{eq:last-exit}, we have 
		\(
		\overline{\Es}(\Gamma _{n}\cup \{0\})=G_{n}^{-1}\,Z_{n}.
		\)
		It is easy to see that $1\leq G_{n}\leq 8$. Furthermore,  using the transience of $S$, we can see that
		with probability one \(	G_{\infty }:=\lim_{n\rightarrow \infty }G_{n}\) exists and equals \(G_{\mathbb{Z}^{4}\setminus
			\Gamma }(0,0)\).
		\begin{theorem}
			\label{maintheoremjune} For every $0 \leq r,s < \infty$, there exists $%
			0<c_{r,s}<\infty, $ such that 
			\begin{equation*}
				\lim_{n\to\infty}n^{r/3}\mathbb{E}\left[Z_n^r \, G_n^{-s}\right] =c_{r,s}.
			\end{equation*}
			Moreover, $c_{3,2} = \pi^2/24$.
		\end{theorem}
		Our methods do not compute the constant $c_{r,s}$ except in the  case $%
		r=3,s=2$ (and the trivial case $r=s=0$). 
		
		The proof of Theorem~\ref{maintheoremjune}, which is the technical bulk of this paper,  requires several steps which we will
		outline now. For the remainder of this paper we fix $r>0$ and allow
		constants to depend on $r$. If $n \in \N$, we let
		\begin{equation}\label{def-pn}
			p_n = \mathbb{E}\left[Z_n^r \right], \;\;\;\;\;
			\hat p_n = \mathbb{E}\left[
			Z_n^3 \, G_n^{-2}\right],
		\end{equation} 
		\begin{equation}\label{eq:def-phi}
			h_n=\E[Y_n]=\E[H(\eta^n)], \;\;\;\;\;
			\phi_n = \prod_{j=1}^n  e^{- h_j}.
		\end{equation}
		
		In Section~\ref{prelsec}, we review and prove some basic estimates on simple random walk that, in particular,  give $h_n= O(n^{-1})$, and hence,
		\begin{equation*}
			\phi_{n} = \phi_{n-1} \, e^{-h_n} = \phi_{n-1} \, %
			\left[1 + O\left(n^{-1} \right) \right].
		\end{equation*}
		In  Section \ref{nicesec}, we revisit the theory of slowly recurrent sets in \cite{LSlowly,3rdmoment} and obtain quantitive estimates on the escape probablity of slowly recurrent sets under a mild assumption (see Definition~\ref{def:slow}). Using these estimates,  we prove two propositions in Section~\ref{subseqsec}. 
		The first one controls  $p_n/p_{n+1}$:
		\begin{proposition}	\label{nov22.cor5}
			\(p_{n+1} = p_{n} \, \left[1 -O\left({\log^4n/n} \right) \right].\)
		\end{proposition}
		The second one gives a good estimate along the subsequence $\{n^4\}$.
		Let $\tilde\eta_n$ denote the (forward) loop-erasure of $%
		S[\sigma_{(n-1)^4 + (n-1)}, \sigma_{n^4 - n}]$. For $m<n$ we let 
		$A(m,n) $ be the discrete  annuli defined by
		\begin{equation*}
			A(m,n) = C_n \setminus
			C_{m} = \{z \in C_n: |z| \geq e^{m}\}.
		\end{equation*}
		Let \(		\tilde \Gamma_n = \tilde \eta_n \cap A( (n-1)^4+ 4(n-1), n^4 - 4n) \) and \( \tilde h_n = \mathbb{E}\left[H(\tilde \Gamma_n) \right].\)
		
		\begin{proposition}\label{subsequenceprop}  There exists $c_0<\infty$ such that
			\begin{equation*}
				p_{n^4} = \left[c_0 + O\left(n^{-1} \right) \right]\, \exp \left\{-r
				\sum_{j=1}^n \tilde h_j \right\}.
			\end{equation*}
		\end{proposition}
		In Section~\ref{sec:full}, in order to get rid of the subsequence $\{n^4\}$, we prove
		\begin{proposition}
			\label{nov22.prop3} There exists $c< \infty, u > 0$ such that 
			\begin{equation*}
				\Big| \tilde h_n- \sum_{(n-1)^4 < j\le n^4} h_j  \Big| \leq \frac{c}{n^{1+u}}.
			\end{equation*}
		\end{proposition}
		Proposition~\ref{nov22.prop3} intuitively says that 
		if the random walk hits $\tilde \Gamma_n$ then
		it does so by  hitting exactly one of $\eta^j$'s.
		This proposition, which is key for proving our main
		result, does not follow from the work in
		\cite{3rdmoment}.  Some of the earlier propositions
		have been improved here in order to be able to
		establish this.
		To rigorously prove this, we need a frequency  estimate on cut points of SRW  and a large deviation estimate on the harmonic measure of the range of SRW obtained in Section~\ref{loopfreesec} and~\ref{sec:deviation} respectively.  Proposition~\ref{nov22.cor5}---\ref{nov22.prop3}   and  readily yield
		the folloing.
		\begin{proposition}
			\label{prop:rs} \label{greenprop} For every $r,s$ there exists constant $c^{\prime
			}_{r,s},u>0$ such that 
			\begin{equation*}
				\mathbb{E}\left[Z_n^r \, G_n^{-s}\right] = c_{r,s}^{\prime }\, \phi_n^r %
				\left[1 + O(n^{-u}) \right].
			\end{equation*}
			In particular, there exists a constant $c'_{3,2}>0$ such that 
			\begin{equation}  \label{nov30.1}
				\hat p_n = c_{3,2}^{\prime }\, \left[1 + O(n^{-u}) \right] \,\exp
				\left\{-3\sum_{j=1}^n h_j \right\},
			\end{equation}
		\end{proposition}
		In  Section~\ref{constantsec},  we use a path decomposition  to study the long-range intersection of an SRW and a LERW and show in Proposition~\ref{prop:3.26}  that 
		there exists $u > 0$ such that 
		\begin{equation}\label{eq:hn}
			h_n = \frac 8 {\pi^2}\, \hat p_n + O(n^{-1-u}).
		\end{equation}
		Combined with \eqref{nov30.1}, this gives  that 
		the limit 
		\begin{equation}\label{eq:exist}
			\lim_{n \rightarrow \infty} \left[\log \hat p_n + \frac{24}{\pi^2} 
			\sum_{j=1}^n \hat p_j \right] 
		\end{equation}
		exists and is finite.
		Note that $\lim_{n\to\infty}\hat p_{n+1}/\hat p_n= 1$ (see Proposition~\ref{prop:rs}). In Section~\ref{exactsec} we prove an elementary lemma  (see Lemma \ref{nov21.lemma2}) on  sequences asserting that this combined with \eqref{eq:exist} assures that $\lim_{n\to\infty} n\hat p_n= \pi^2/24$. 
		Now \eqref{nov30.1} and~\eqref{eq:hn} imply that $\lim_{n\to\infty}3nh_n=1$ and
		\begin{equation}  \label{jun24.1}
			\phi_n^3 = \exp \left\{-3\sum_{j=1}^n h_j \right\} \sim \frac{c}{n} \quad  \text{for some constant}\; c>0.
		\end{equation}
		This combined with Proposition~\ref{prop:rs} concludes the proof of Theorem~\ref{maintheoremjune}.  This already implies Theorem~\ref{thm:const} by changing scales. 
		The proof of Theorem~\ref{thm:inter} will be explained in Section~\ref{sec:two-sided}.
		
		\subsection{Two sided LERW}\label{subsec:two-sided}
		In \cite{CCLERW}, the first author author proved the existence of \emph{two-sided loop-erased random walk} in $\Z^d$ for $d\ge 5$.
		\begin{theorem}[\cite{CCLERW}]
			Given $d\ge 5$, consider the sample of  LERW in $\Z^d$, denoted by $\{\eta_i\}_{i\ge 0}$. The $n\to\infty$ limit of
			$\{\eta_{n+i}-\eta_n \}_{-k\le i\le k}$ exists for any $k\in \N$, which defines an ergodic random path $\{\tilde{\eta}_i\}_{i\in \Z}$ in $\Z^d$  called the two-sided LERW.
			\label{thm:two-sided}
		\end{theorem}
		The proof of Theorem~\ref{thm:two-sided} crucially replies on the existence of global cut points for SRW  in $\Z^d$ for $d\ge 5$, which is not true for $d\le 4$.
		As an  application of results in Section~\ref{subsec:nis}, we extend the existence of the two-sided LERW to $d=4$
		in Section~\ref{sec:two-sided}. Moreover, $X_\infty$ in Theorem~\ref%
		{thm:inter} is the Radon-Nikodym derivative between the two-sided LERW
		restricted to non-negative times and the usual LERW.  The existence for $d=2,3$ was recently established by the first author author in \cite{lawler-twoside}.
		A big difference in $d<4$ compared to $d\ge 4$ case is that the marginal
		distribution of one side of the path is \emph{not} absolutely continuous
		with respect to the usual LERW. 
		
		Our results addresses the $d=4$ case of Conjecture 15.12 in \protect\cite{BLPS} by  Benjamini-Lyons-Peres-Schramm, which asserts the existence of the two-sided uniform spanning tree in $\Z^d$. This is immediate from Wilson's algorithm~\cite{Wil96} that connects LERW and uniform spanning tree (see Section~\ref{sec:USF}). 
		\subsection{A spin field from USF} \label{subsec:field}
		As an application of Theorem~\ref{thm:long-range}, we will construct a sequence
		of random fields on the integer lattice $\mathbb{Z}^{d}$ ($d\geq 4$) using
		uniform spanning tree  and show that they converge in distribution to the bi-Laplacian field
		(Theorem \ref{BGF}).

		For each positive integer $n$, let $N=N_{n}=n(\log
		n)^{1/4}$. Let $A_N=\{x\in \mathbb{Z}^{d}:|x| <N\}$. We will construct
		a $\pm 1$ valued random field on $A_N$ as follows. Recall that a \emph{%
			wired spanning tree} on $A_N$ is a tree on the graph $A_N\cup \{\partial
		A_N\}$ where we have viewed the boundary $\partial A_N$ as
		\textquotedblleft wired\textquotedblright\ to a single point. Such a tree
		produces a \emph{spanning forest} on $A_N$ by removing the edges connected
		to $\partial A_N$. We define the uniform spanning forest (USF) on $A_N$
		to be the forest obtained by choosing the wired spanning tree of $A_N\cup
		\{\partial A_N\}$ from the uniform distribution. (Note this is not the
		same thing as choosing a spanning forest uniformly among all spanning
		forests of $A_N$.) We now define the random field on (a rescaling
		of) $\mathbb{Z}^{d}$. Let $a_{n}$ be a
		sequence of positive numbers (we will be more precise later).
		
		\begin{itemize}
			\item Choose a USF on $A_N$. This partitions $A_N$ into (connected)
			components.
			
			\item For each component of the forest, flip a fair coin and assign each
			vertex in the component value $1$ or $-1$ based on the outcome. This gives a
			field of spins $\{Y_{x,n}:x\in A_N\}$. If we wish we can extend this to a
			field on $x\in \mathbb{Z}^{d}$ by setting $Y_{x,n}=0$ for $x\not\in A_N$.
			
			\item Let $\phi_n(x) = a_n \, Y_{nx,n}$ which is a field defined on $L_n
			:=n^{-1} \, \mathbb{Z}^d$.
		\end{itemize}
		
		This random function is constructed in a manner similar to the Edward-Sokal
		coupling of the FK-Ising model \cite{Gri}. That coupling says that we can
		obtain the Ising model on $\mathbb{Z}^{d}$ by first sample a random
		configuration $\omega \in \left\{ 0,1\right\} ^{\mathbb{Z}^{d}}$ according
		to the so-called random cluster measure, and then flip a fair coin and
		assign each component of $\omega $ value $1$ or $-1$ based on the outcome.
		The way we construct $\phi _{n}$ is similar to the Ising model except that
		we replace the random cluster measure by the USF measure on $\mathbb{Z}^{d}$.
		
		It is known that the Ising model has critical dimension $d=4$, in the sense
		that mean field critical behaviors are expected for $d\geq 4$ but not for $%
		d\leq 3$. In particular, it is believed when $d\geq 4$ the scaling limit of
		Ising model is a $d-$dimensional Gaussian Free Field (GFF). For $d\geq 5$ this GFF limit is proved by Aizenman \cite{Aiz}, while
		the critical case $d=4$ is still open. 
		Our theorem  below  asserts  that the random field $%
		\phi _{n}$ we construct has critical dimension $d=4$, and for $d\geq 4$, we
		can choose the scaling $a_{n}$ such that $\phi _{n}$ converges to the
		bi-Laplacian Gaussian field on $\mathbb{R}^{d}$. Note that when $d=4$, a
		bi-Laplacian Gaussian field is log-correlated.
		
		If $h\in C_{0}^{\infty }(\mathbb{R}^{d})$, we write 
		\begin{equation*}
			\langle h,\phi _{n}\rangle =n^{-d/2}\sum_{x\in L_{n}}h(x)\,\phi _{n}(x).
		\end{equation*}
		
		\begin{theorem}
			$\;\label{BGF}$
			
			\begin{itemize}
				\item If $d \geq 5$, there exists $a>0$ such that if $a_n = a \, n^{(d-4)/2}$%
				, then for every $h_1,\ldots,h_m \in C_0^\infty(\mathbb{R}^d)$, the random
				variables $\langle h_j, \phi_n \rangle$ converge in distribution to a
				centered joint Gaussian random variable with covariance 
				\begin{equation*}
					\int \, \int h_j(z) \, h_k(w) \, |z-w|^{4-d} \, dz \, dw .
				\end{equation*}
				
				\item If $d=4$, if $a_{n}=\sqrt{3\log n}$, then for every $h_{1},\ldots
				,h_{m}\in C_{0}^{\infty }(\mathbb{R}^{d})$ with 
				\begin{equation*}
					\int h_{j}(z)\,dz=0,\;\;\;j=1,\ldots ,m,
				\end{equation*}%
				the random variables $\langle h_{j},\phi _{n}\rangle $ converge in
				distribution to a centered Gaussian random variable with variance 
				\begin{equation*}
					-\int \,\int h_{j}(z)\,h_{k}(w)\,\log |z-w|\,dz\,dw.
				\end{equation*}
			\end{itemize}
		\end{theorem}
		\begin{remark}
			$\;$
			
			\begin{itemize}
				\item Gaussian fields on $\mathbb{R}^d$ with correlations as in Theorem~\ref{BGF} is called 
				$d$-dimensional bi-Laplacian Gaussian field (see \cite{FGF}).  
				
				\item For $d=4$, we could choose the cutoff $N = n(\log n)^\alpha$ for any $%
				\alpha > 0$. We choose $\alpha=\frac14$ for concreteness. For $d > 4$, we
				could do the same construction with no cutoff ($N = \infty$) and get the
				same result.
			\end{itemize}
		\end{remark}
		By Wilson's algorithm, the two-point correlation function of the field $\phi_n$ is proportional to the intersection  probability of an SRW and a LERW stopped upon hitting $\partial A_N$. Therefore Theorem~\ref{BGF}  essentially follows from Theorem~\ref{thm:long-range}. In particular, $G^2_n(w)$  there is the discrete
		biharmonic function that gives the covariance structure of the bi-Laplacian random field in the scaling limit. 
		We will give the full proof of Theorem~\ref{BGF} in Section~\ref{sec:field}, where we only deal with the case $d=4$. The $d\ge 5$ case can be proved
		in the same way but is much easier (see  \cite{sunwu} for a detailed argument).

		\section{Preliminaries}\label{sec:pre}
		In this section we recall and prove necessary lemmas about  SRW  and LERW which will be used frequently in the rest of the paper.  Throughout this section we retain the notations in Section~\ref{subsec:outline}.
		\subsection{Basic notations} 
		Given a vertex set $V\subset \Z^d$, $\partial V$ is the set of vertices on $\Z^d\setminus V$ who have  a neighbor in $V$, and $\overline V = V \cup \partial V$.
		A function $\phi$ on $\overline V$ is called \emph{harmonic} on $V$ if for each  $v\in V$,
		we have $\E^v[ \phi(S_1)]=\phi(v)$.
		When we say ``$\phi$ is harmonic on $V$'' then
		it is implicit that $\phi$ is defined on $\overline V$.
		
		We will use $c$ and $C$ to represent constants which may vary line by line. 
		We use the asymptotic notion that two nonnegative
		functions $f(x),g(x)$ satisfy $f\lesssim g$ if
		there exists a constant $C>0$ independent of $x$  such that $f(x)\le Cg(x)$. We write $f\gtrsim g$ if $g\lesssim f$ and write $f\asymp g$ if $f\lesssim g$ and $g\gtrsim f$. 
		Given a sequence $\{a_n\}$ and a nonnegative sequence $\{b_n\}$, we write $a_n\sim b_n$ if $\lim_{n\to\infty}a_n/b_n=1$. We write $a_n=O(b_n)$ if $|a_n|\lesssim b_n$.
		We write $a_n=o(b_n)$ if $\lim_{n\to\infty}|a_n|/b_n=0$. 
		When $\{b_n\}=\{1\}$, we may write $o(1)$ as $o_n(1)$ to indicate the dependence on $n$.
		
		We say that a sequence $\{\epsilon_n\}$ of positive numbers is 
		\emph{fast decaying} if it decays faster than every power of $n$, that is $n^k\epsilon_n=o_n(1)$ for 
		for every $k > 0$. We will write $\{\epsilon_n\}$ for fast decaying
		sequences. As is the convention for constants, the exact value of $%
		\{\epsilon_n\}$ may change from line to line. We will use implicitly the
		fact that if $\{\epsilon_n\}$ is fast decaying then so is $%
		\{\epsilon_n^{\prime }\}$ where \(\epsilon_n^{\prime }= \sum_{m \geq n} \epsilon_m.\)

		\subsection{Estimates for simple random walk on \texorpdfstring{$\Z^4$}{Lg}}
		\label{prelsec} In this subsection we provide facts about
		simple random walk in $\mathbb{Z}^4$, which will be frequently used
		in the rest of the paper.  
		We first recall the following facts about
		intersections of random walks in $\mathbb{Z}^4$ (see \cite{L4dint2,LSlowly}).
		\begin{proposition}
			\label{sep8.prop1} There exist $0<c_{1}<c_{2}<\infty $ such that the
			following is true. Suppose $S$ is a simple random walk on  $\mathbb{Z}^{4}$ starting at 0. Then 
			\begin{align*}
				\frac{c_{1}}{\sqrt{\log n}}& 
				\leq \Prob\left\{ S[0,n]\cap S[n+1,\infty ]=
				\eset\right\} \\& 
				\leq \Prob\left\{ S[0,n]\cap S[n+1,2n]=\eset\right\} 
				\leq \frac{c_{2}}{\sqrt{\log n}},
			\end{align*}%
			and if $2 \leq \alpha \leq n$,
			\begin{equation}\label{eq:srw-long-range}
				c_{1}\,\frac{\log \alpha }{\log n}\leq \Prob\left\{ S[0,n]\cap
				S[n\,(1+\alpha ^{-1}),\infty )\neq \eset\right\} \leq c_{2}\,\frac{\log
					\alpha }{\log n} .
			\end{equation}%
			Moreover, if $S^{1}$ is an independent simple random walk starting at $z\in 
			\mathbb{Z}^{4}$, 
			\begin{equation}
				\Prob\left\{ S[0,n]\cap S^{1}[0,\infty )\neq \eset\right\} \leq c_{2}\,\frac{%
					\log a}{\log n},   
				\label{sep10.1}
			\end{equation}%
			where \(	a=\max \left\{ 2,\sqrt{n}/|z|\right\} .\)
		\end{proposition}
		
		An important corollary of Proposition~\ref{sep8.prop1}  is that 
		\begin{equation}  \label{nov22.7}
			\sup_{n} n \, \mathbb{E}\left[H(\eta^n) \right] \leq \sup_n n \, \mathbb{E}%
			\left[H(\omega_n) \right] < \infty ,
		\end{equation}
		and hence 
		\begin{equation*}
			\exp\left\{-\mathbb{E}\left[H(\eta^n)\right]\right\} = 1 - \mathbb{E}\left[%
			H(\eta^n)\right] + O(n^{-2}),
		\end{equation*}
		It follows that if $\phi_n$ is defined as in 
		\eqref{eq:def-phi} and $m < n$, then
		\begin{equation}\label{eq:phi-asym}
			\phi_n = \phi_m \left[ 1 + O(m^{-1})\right] \prod_{j=m+1}^n \left[1 - 
			\mathbb{E}\left[H(\eta^j)\right] \right].
		\end{equation}
		
		\begin{corollary}	\label{cor:har} 
			There exists $c < \infty$ such that if $n\in \N$,
			$\alpha \geq 2$,  $0 < u < 1$,  
			$m = m_{n,\alpha} = (1+ \alpha^{-1}) \, n$ and \(Y_{n,\alpha} = \max_{j \geq m} H(S_j, S[0,n])\), then
			\begin{equation*}
				\Prob\left\{Y_{n,\alpha} \geq \frac{\log \alpha} { ({\log n})^u} \right\} \leq \frac{c}{(\log n)^{1-u}}.
			\end{equation*}
		\end{corollary}
		
		\begin{proof}  We fix $n,\alpha,u$ and define the stopping time $\tau$ by
			\[  \tau = \min\left\{j \geq m: H(S_j,S[0,n] )\geq  
			\frac{\log \alpha} {   ({\log n})^u}\right\}. \]
			The strong Markov property of $S$ and the definition of $\tau$ imply that
			\[  \Prob\left\{S[0,n] \cap S[m,\infty)
			\neq \eset \mid   \tau < \infty \right \} \geq  
			\frac{\log \alpha} {   ({\log n})^u}, \]
			and hence, using \eqref{eq:srw-long-range},
			\[  \Prob\{\tau < \infty\} \leq \frac{   ({\log n})^u}{\log \alpha} \, \Prob\{S[0,n] \cap S[m,\infty)
			\neq \eset \} \leq \frac{c}{(\log n)^{1-u}}.\qedhere\]
		\end{proof}
		
		\begin{lemma}
			\label{standardlemma} {There exists $c>0$
				such that the followind holds for all $n \in \N$.
				\begin{itemize}
					\item if $\phi$ is a positive (discrete) harmonic function on $C_n$ and $x \in C_{n-1}$,
					\begin{equation}  \label{harmonic}
						\left|\log[\phi(x)/\phi(0)] \right| \leq c \, |x| \, e^{-n}.
					\end{equation}
					
					\item  If $m < n$,  $V \subset C_{n-m}$, and $\tilde V
					= V \setminus C_{n-m-1}$,
					\begin{equation}  \label{nov15.1}
						c^{-1} \, e^{-2m} \leq 	\Prob\{S[\sigma_n,\infty) \cap C_{n-m} \neq \eset \mid \F_n \} \leq c \,
						e^{-2m};
					\end{equation}	
					\begin{equation}  \label{nov15.1.improve}
						c^{-1} \, e^{-2m} H(\tilde V)
						\leq 		\Prob\{S[\sigma_n,\infty) \cap V \neq \eset \mid \F_n \} \leq c \, e^{-2m} H(V).
					\end{equation}
					\begin{equation}  \label{nov15.1.improve3}
						\Es(V) \geq 
						\Prob\{S[0,\sigma_n] \cap V = \eset \}
						\, [1-c \,e^{-2m} \, H(V)] .
					\end{equation}
				\end{itemize}}
			\end{lemma}
			
			\begin{proof} The inequalities \eqref{harmonic}
				and \eqref{nov15.1} are standard estimates, see, e.g.,
				\cite[Theorem 6.3.8, Proposition 6.4.2]{lawler-limic}.
				The Harnack principle \cite[Theorem 6.3.9]
				{lawler-limic} shows that $H(z,V) \asymp
				H(z',V)$ for $z,z'\in \partial C_{n-m+1}$, and
				by stopping at time $\sigma_{n-m+1}$ we see
				that
				\[  H(V) \geq \min_{z \in \partial C_{n-m+1}
				}  H(z,\bar V) .\]
				This combined with  \eqref{nov15.1} and the strong
				Markov property
				gives the upper bound in
				\eqref{nov15.1.improve}.  To get the lower
				bound, one uses the
				Harnack principle to see that for $z
				\in  \partial C_{n-m+1}$, $H(z,V^+)
				\asymp H(V^+)$ and $H(z, V \setminus V^+)
				\asymp H(V\setminus V^+)$, where
				\[  V^+ = \tilde V \cap \{(z_1,\ldots,z_4)
				\in \Z^4: z_1 \geq 0\}.\]
				Finally, \eqref{nov15.1.improve3} follows from
				the upper bound in \eqref{nov15.1.improve} and
				the strong Markov property.
			\end{proof}

			\begin{lemma}
				\label{nov21.lemma1} Let $U_n$ be the event that there exists $k \geq
				\sigma_n$ with 
				\begin{equation*}
					\LE(S[0,k]) \cap C_{n - \log^2 n} \neq \hat S[0,\infty)\cap C_{n - \log^2 n }.
				\end{equation*}
				Then $\Prob(U_n)$ is fast decaying.
			\end{lemma}
			
			\begin{proof} 
				By the loop-erasing process, we can see that
				the event $U_n$ is contained in  the event that either
				\[     S[\sigma_{n - \frac 12 \log^{2} n},\infty) \cap
				C_{n - \log^2 n}  \neq \eset   \qquad\textrm{or}\qquad S[\sigma_n,\infty) \cap C_{n- \frac 12 \log^{2} n}
				\neq \eset.\]
				The probability that either of these happens is fast decaying
				by \eqref{nov15.1}. 
			\end{proof}
			The next proposition gives a quantitative estimate on
			the slowly recurrent nature of a simple
			random path in $\Z^4$.    
			\begin{proposition}
				\label{nov15.prop1} If $\Lambda(m,n) = S[0,\infty) \cap A(m,n),$ then the
				sequences 
				\begin{equation*}
					\Prob\left\{H[\Lambda(n-1,n)] \geq \frac{\log^2n}{n} \right\}
					\qquad \textrm{and}\qquad 
					\Prob\left\{H(\omega_n) \geq \frac{\log^4 n}{n} \right\}
				\end{equation*}
				are fast decaying.
			\end{proposition}
			
			\begin{proof}
				For any $z\in \Z^4$, let $S^z$ be a simple random walk starting from $z$ independent of $S$.  
				Let 
				$  \Lambda^z_j= \Lambda^z_j(n-1,n)
				= S^z[0,j] \cap A(n-1,n)$ for $j\in \N\cup \{\infty\}$.
				By the definition of $H$ and Proposition~\ref{sep8.prop1}, there exists a positive constant $c$  such that for each $z$ with $|z|\ge e^{n-1}$,
				\[
				\E[H(\Lambda^z_\infty)]  
				=\P\{ S[0,\infty) \cap S^z[0,\infty) \cap A(n-1,n) \neq \emptyset\}\le  \frac{c}{4n},
				\] 	
				From now on we assume $|z|\ge e^{n-1}$.  Then by the Markov inequality,
				\begin{equation}\label{eq:Markov}
					\Prob\left\{H(\Lambda^z_\infty)   \geq 
					\frac{c}{2n}  \right\} \leq \frac 12.
				\end{equation} 
				For each  $k \in \N$, let 	$\tau_k=\inf\left \{j: H(\Lambda^z_j ) \ge  {ck}/n\right\}$. 
				On the event $\tau_k<\infty$, we have  $H(\Lambda^z_{\tau_k-1})< ck/n$ and $S^z_{\xin{\tau_k}}  \in A(n-1,n)$.
				Since 
				$$H(A \cup B)\leq H(A) + H(B) \qquad\textrm{for any} \qquad A,B\subset \Z^4,$$  for $n$ sufficiently large, we have 
				$$H(\Lambda_{\tau_k})\le H(\Lambda^z_{\tau_k-1})+H(S^z_{\tau_k}) \leq   c(k+ \frac 12)/n.$$ 
				Moreover, combined with \eqref{eq:Markov} we see that for $n$ sufficiently large,
				\begin{align*}
					\Prob\{\tau_{k+1} < \infty \mid \tau_k < \infty\}&\le \sum_{w\in A(n-1,n)}\P[S^z_{\tau_k} =w|\tau_k<\infty]  \P[H(\Lambda^w_\infty)\ge \frac{c}{2n}]\\
					&\le \frac12\sum_{w\in A(n-1,n)}\P[S^z_{\tau_k} =w|\tau_k<\infty] =\frac12.
				\end{align*}
				Therefore $  \Prob\{\tau_{k} <\infty\} \leq 2^{-k}.$
				Setting $k =  \lfloor c^{-1} \log^2n\rfloor$, we see  that the first sequence in Proposition~\ref{nov15.prop1} is fast decaying.
				
				For the second sequence, note that on the event $\{H(\omega_n) \geq \log^4 n/{n} \}$, either $\omega_n \not\subset A(n-\log^2 n,  n)$ or there exists a $j\in [n-\log^2n,n]$ such that $H[\Lambda(j-1,j)] \geq {\log^2n}/{n} $. 
				We use  \eqref{nov15.1} to see that 
				$\Prob\{\omega_n \not\subset A(n-\log^2 n,  n)\}$
				is fast decaying.
			\end{proof}

			\subsection{Loop-free times}
			
			\label{loopfreesec} One of the technical nuisances in the analysis of the
			loop-erased walk is that if $j < k$, it is not necessarily the case that 
			\begin{equation*}
				\LE(S[j,k]) = \LE(S[0,\infty)) \cap S[j,k].
			\end{equation*}
			However, this is the case for special times which we call
			\emph{loop-free} times.
			We say that $j$ is a \emph{(global) loop-free time} if 
			\begin{equation*}
				S[0,j] \cap S[j+1,\infty) = \emptyset .
			\end{equation*}
			Proposition \ref{sep8.prop1} shows that the probability that $j$ is
			loop-free is comparable to $(\log j)^{-1/2}$. From the definition of
			chronological loop erasing we can see the following. If $j < k$ and $j,k$
			are loop-free times, then for all $m \leq j < k \leq n$, 
			\begin{equation}  \label{loopfree}
				\LE\left(S[m,n]\right) \cap S[j,k] = \LE\left(S[0,\infty) \right) \cap S[j,k]
				= \LE\left(S[j,k]\right).
			\end{equation}
			
			It will be important for us to give upper bounds on the probability that
			there is \emph{no} loop-free time in a certain interval of time. If $m \leq
			j < k \leq n$, let $I( j,k;m,n)$ denote the event that for all $j \leq i
			\leq k-1$, 
			\begin{equation*}
				S[m,i] \cap S[i+1,n] \neq \emptyset .
			\end{equation*}
			Proposition \ref{sep8.prop1} gives a lower bound on $\Prob\left[I(n,2n;0,3n)%
			\right]$, 
			\begin{equation*}
				\Prob\left[I(n,2n;0,3n)\right] \geq \Prob\{S[0,n] \cap S[2n,3n] \neq \eset %
				\} \asymp \frac{1}{\log n}.
			\end{equation*}
			The next lemma shows that
			\begin{equation}\label{eq:loop-free}
				\Prob\left[I(n,2n;0,3n)\right] \asymp 1/\log n
			\end{equation}
			by giving the matching upper bound.
			
			\begin{lemma}\label{lem:loop-free}
				There exists $c < \infty$ such that \(\Prob\left[I(n,2n;0,\infty)\right] \leq c/\log n\).
			\end{lemma}
			
			\begin{proof}  Let $E = E_{n}$ denote the complement
				of $I(n,2n;0,\infty)$. We need to show that $\Prob(E)
				\geq 1 - O(1/\log n)$.
				
				Let $k_n =\lfloor n/(\log n)^{3/4}\rfloor$ and let $A_i = A_{i,n}$ be the event
				that  
				\[  A_i 
				= \left\{             S[n + (2i-1)k_n
				,n+2ik_n] \cap S[n+ 2ik_n + 1,
				n+(2i+1)k_n] = \eset\right\} .\]
				and consider the events $A_1,A_2,\ldots, A_\ell$
				where $\ell =\lfloor (\log n)^{3/4}/4\rfloor $.  These are  $\ell$
				independent 	events each with probability  greater than $c\,(\log n)^{-1/2}$ by Proposition \ref{sep8.prop1}. Therefore
				\[1-   \Prob(A_1\cup \cdots  \cup A_\ell)= \prod_{i=1}^{\ell} [1-\P(A_i)] \le  \exp\{-O((\log n)^{1/4})\} = o( \frac{1}{\log^3n}).\]
				Let $B_i = B_{i,n}$ be the event \(\left\{S[0,n + (2i-1)k_n]\cap S[n+ 2ik_n,\infty) = \eset\right\}\).
				By \eqref{eq:srw-long-range} in Proposition \ref{sep8.prop1},  $\Prob(B_i^c)
				\leq  c\, \log \log  n/\log n$. Therefore 
				\[   \Prob (B_1 \cap \cdots \cap B_\ell)
				\geq 1 - \frac{c\ell\, \log \log n}{\log n}
				\geq 1- O\left(\frac{\log \log n}{(\log n)^{1/4}}
				\right).\]
				For $1\le i\le \ell$, on the event $A_i\cap (B_1 \cap \cdots \cap B_\ell) $ the time $2ik_n$ is  loop-free, hence $E$ occurs.
				Therefore	
				\begin{equation}  \label{sep7.1}
					\Prob(E) \geq \Prob \left[(A_1 \cup \cdots \cup A_\ell)
					\cap (B_1 \cap \cdots \cap B_\ell)\right]
					\geq 1 - O\left(\frac{\log \log n}{(\log n)^{1/4}}
					\right).
				\end{equation}
				This is a good estimate, but we need to improve on it.
				
				Let $C_j, j=1.,\ldots,5, $  denote the independent events
				(depending on $n$) 
				\[ I\left( n\, \left[1 + \frac{3(j-1) + 1}{15}\right],
				n\, \left[1 + \frac{3(j-1) + 2}{15} \right]; n+ \frac{(j-1)n}{5},
				n + \frac{jn}{5}\right). \]
				By \eqref{sep7.1} we see that $\Prob[C_j] \leq 
				o\left(1/{(\log n)^{1/5}}\right)$, and hence 
				\[   \Prob(C_1 \cap \cdots \cap C_5)  \leq o\left(\frac{1}{\log n}\right).\]
				Let $D = D_n$ denote the event that at least one of the following
				ten things happens:
				\[  S\left[0,  n\, \left(1 + \frac{j-1 }{5} 
				\right) \right]\cap S\left[ n\, \left(1 + \frac{3(j-1) + 1}{15}\right) , \infty
				\right) \neq \eset , \;\;\; j=1,\ldots,5;\]
				\[ S\left[0,  n\,  \left(1 + \frac{3(j-1) + 2}{15}\right) 
				\right] \cap S\left[ n\, \left(1 + \frac{j}{5}\right) , \infty\right)
				\neq \eset, \;\;\; j=1,\ldots,5.\]
				Each of these events has probability comparable to  $ 1/\log n $ and hence
				$\Prob(D) \asymp 1/\log n$.  Also, 
				\[    I(n,2n;0,\infty)  \subset (C_1 \cap \cdots \cap C_n) \cup \, D.\]
				Therefore, $\Prob \left[I(n,2n;0,\infty) \right] \leq c/\log n$. 
			\end{proof}
			
			\begin{corollary}
				$\;$ \label{nov21.cor1}
				
				\begin{enumerate}
					\item There exists $c < \infty$ such that if $0 \leq j \leq j+ k \leq n$,
					then 
					\begin{equation}\label{eq:loop-free2}
						\Prob\left[I(j,j+k;0,n)\right] \leq \frac{c\,\log (n/k)} {\log n}. 
					\end{equation}
					
					\item Given $0 < \delta < 1$, let $I_{\delta,n} := \bigcup_{j=0}^{n-1} I(j,j+\delta n;0,n)$. 
					Then there exists a positive constant  $c$ such that for all $n\in \N$ and $\delta\in (0,1)$, we have 
					\begin{equation}  \label{jun7.4}
						\Prob\left[I_{\delta,n} \right] \leq \frac{c \, \log(1/\delta)} {\delta \,
							\log n}.
					\end{equation}
					
					\item There exist $c < \infty$ and a positive
					integer $\ell$ such that the following holds
					for all positive integers $n$.
					Let $\tilde I(m,r)$ denote the event that there is no loop-free point $ 
					j$ with $\sigma_m \leq j \leq \sigma_r$,
					and let $k = k_n = \lfloor \log n\rfloor $.
					Then 
					\begin{equation}  \label{jun7.1}
						\Prob\{\tilde I(n-\ell k, n + \ell k) \mid \F_{n-3\ell k} \} \leq c/n .
					\end{equation}
				\end{enumerate}
			\end{corollary}
			
			\begin{proof}$\;$
				
				\begin{enumerate}
					
					\item  It suffices to prove under the assumption that $k \geq n^{1/2}$.
					Note that $I(j,j+k;0,n)$ is contained in the union
					of the following two events:
					\[ I(j,j+k; j-k, j+2k),\]
					\[ \{S[0,j-k] \cap S[j,n] \neq \eset \}, \qquad \textrm{and}\qquad \{S[0,j] \cap S[j+k,n] \neq \eset \}.\]
					Since $k \geq n^{1/2}$, the probability of the
					first event is $O(1/\log n)$ by Lemma~\ref{lem:loop-free}. By \eqref{eq:srw-long-range},  the probabilities of the second two events are $O(\log (n/k)/\log n)$.
					This gives \eqref{eq:loop-free2}.
					
					\item By \eqref{eq:loop-free2}, $\P\{ I( i\delta n/3, (i+1)\delta n/3 ;0,n) \} =O(\log(\delta^{-1})/\log n)$ for all $0\le i \le \lceil3/\delta\rceil$.
					Now \eqref{jun7.4} follows from the fact that  $I_{\delta,n}$ can be covered by these $I( i\delta n/3, (i+1)\delta n/3 ;0,n)$'s.
					
					\item {We will first consider
						walks starting at $z \in \p C_{n - 3 \ell k} $
						(with constants independent of $z$).  Let
						$E_n$ be the event
						\[  E_n = 
						\{\sigma_{n - \ell k} \leq e^{2n}
						\leq 2 \, e^{2n}  \leq \sigma_{n + \ell k},
						\;\;S[\sigma_{n-\ell k},\infty) \cap C_{n-2\ell k}  = \eset\}.\]
						Using 			\eqref{jul18.1} and  \eqref{nov15.1},
						we can choose $\ell$
						sufficiently large so that
						$\P(E_n)\ge 1-1/n$. 
						On the event $E_n$, we have $\tilde I(n-\ell k,n+\ell k)\subset  I(e^{2n},2e^{2n};0,\infty)$.
						Hence, by 
						Lemma~\ref{lem:loop-free}, we have $\P[\tilde I(n-\ell k,n+\ell k)] \le O(n^{-1})$.	
						
						More generally, if we start the random walk at the
						origin, stop at time $\sigma_{n - 3\ell k	}$, and then
						start again, we can use the result in the previous
						paragraph.  Since $S[\sigma_{n-\ell k},\infty) \cap C_{n-2\ell k}  = \eset$ on the event $E_n$, attachment of the initial part of
						the walk up to $\sigma_{n - 3\ell k	}$ will not effect
						whether a time after $\sigma_{n - \ell k}$  is loop-free. This concludes the proof.
						\qedhere }
				\end{enumerate}

			\end{proof}

			\subsection{Green function estimates}\label{sec;Green}
			Recall the Green function $G(\cdot,\cdot)$ on $\Z^4$ and $G_n(\cdot,\cdot)$ defined in Section~\ref{subsec:nis}.
			We write $G(x)=G(x,0)=G(0,x)$ and $G_n(x)=G_n(x,0)=G(0,x)$. 
			As a standard estimate (see \cite{lawler-limic}), we have
			\begin{equation}\label{eq:Green0}
				G(x) = \frac{2}{\pi^2 \, |x|^2} + O(|x|^{-4}), \;\;\;\; |x | \rightarrow \infty.
			\end{equation}
			Here and throughout we use the convention that if we say that a function on 
			$\mathbb{Z}^{d}$ is $O(|x|^{-r})$ with $r>0$, we still imply that it is
			finite at every point. In other words, for lattice functions, $O(|x|^{-r})$
			really means $O(1\wedge |x|^{-r})$. We do not make this assumption for
			functions on $\mathbb{R}^{d}$ which could blow up at the origin.
			\begin{lemma}\label{lem:Green2}
				For  $w \in C_{n}$, let
				\begin{equation*}
					\hat G^2_n(w) = \sum_{z \in \mathbb{Z}^4} G(0,z) \, G_n(w,z) = \sum_{z \in
						C_n} G(0,z) \, G_n(w,z).
				\end{equation*}Then
				\begin{equation*}
					\hat G^2_n(w) = \frac{8}{\pi^2} \, \left[n - \log |w|\right] \, + O(e^{-n})
					+ O( |w|^{-2} \, \log |w|).
				\end{equation*}
				In particular, if $w \in \partial C_{n-1}$,
				\begin{equation}  \label{dec16.1}
					\hat G^2_n(w) = \frac{8}{\pi^2} + O(e^{-n}).
				\end{equation}
			\end{lemma}

			\begin{proof}  Let $f(x) =   \frac 8{\pi^2} \,   \log |x|$ and note that
				\[    \Delta f(x) =  \frac{2}{\pi^2\, |x|^2} + O(|x|^{-4})
				=   G(x) + O(|x|^{-4}). \]
				where $\Delta$ denotes the discrete Laplacian.  Also,
				we know that 
				\[ 
				f(w)  =   \E^w\left[f(S_{\sigma_n})\right]
				- \sum_{z \in C_n}  G_n(w,z) \,  \Delta f(z)  
				\]
				(this holds for any function $f$). 	Since $e^{n} \leq |S_{\sigma_n}| \leq e^{n}+1$, we have \( \E^w\left[f(S_{\sigma_n})\right] =\frac {8n}{\pi^2}  + O(e^{-n})\). 
				Therefore,
				\[  \sum_{z \in C_n} G_n(w,z) \,  G(z)
				= \frac{8}{\pi^2} \, \left[n - \log |w| \right]
				+ O(e^{-n}) + \epsilon, \]
				where
				\[   |\epsilon| \leq  \sum_{z \in C_n} \, 
				G_n(w,z)
				\,   O(|z|^{-4})
				\leq   
				\sum_{z \in C_n} \, 
				O(|w-z|^{-2}) 
				\,     O(|z|^{-4})  .\]
				
				We split the sum on the right-hand side into three pieces.
				\begin{eqnarray*}
					\sum_{|z| \leq |w|/2}   O(|w-z|^{-2}) 
					\,     O(|z|^{-4}) & 
					\leq & c\, |w|^{-2} \, \sum_{|z| \leq |w|/2} 
					O(|z|^{-4})\\
					& \leq & c\, |w|^{-2} \, {\log |w|}  ,
				\end{eqnarray*}
				\begin{eqnarray*}
					\sum_{|z-w| \leq |w|/2}  O(|w-z|^{-2}) 
					\,    O(|z|^{-4}) & 
					\leq & c\, |w|^{-4} \, \sum_{|x| \leq |w|/2} 
					O(|x|^{-2})\\
					& \leq & c\, 
					|w|^{-2} ,
				\end{eqnarray*}      
				If we let $C_n'$ the the set of $z \in C_n$ with
				$|z| > |w|/2 $ and $|z-w| > |w|/2$, then
				\[  \sum_{z \in C_n'} 
				O(|w-z|^{-2}) 
				\,     O(|z|^{-4})
				\leq   \sum_{|z| > |w|/2    }
				O(|z|^{-6} ) \leq c \, |w|^{-2}.\qedhere \]
			\end{proof}
			\begin{lemma}
				\label{lem:Green} 
				If $1 \leq m < n$ and $x \in C_m$, 
				\[ 	\hat G_n^2(x)- G_n^2(x) =\frac{\pi^2}{2}+  O\left(  e^{m-n}\right). \]  
			\end{lemma}
			
			\begin{proof}  Set $N=e^n$. Using the martingale $M_t = |S_t|^2 - t$, we see that
				\begin{equation} \label{may30.1}
					\sum_{w \in C_n} G_n(x,w) = N^2  - |x|^2 +  O(N) . 
				\end{equation}

				By the strong Markov property, for all $w\in C_n$, 
				\[ \min_{N \leq |z| \leq N+1}
				G(z,x) \leq G(w,x) - G_n(w,x) 
				\leq \max_{N \leq |z| \leq N+1} G(z,x).\]
				Set $\delta = (1 +|x|)/N$. By \eqref{eq:Green0}, we have
				\[      G_n(x,w) = G(x,w)  - \frac{2}{\pi^2 N^2}
				\, [1 + O(e^{m-n})].\]
				Using \eqref{may30.1}, we see that 
				\begin{eqnarray*}
					\lefteqn{\sum_{w \in C_n} G_n(x,w) \, G_n(0,w)}
					\\
					& = & \sum_{w \in C_n}
					\left[ G(x,w)-
					\frac{2}{\pi^2 \, N^2}
					+  O\left(e^{m-n} N^{-2}\right)\right]
					\, G_n(0,w)\\
					& = &  O(\delta) - \frac{2}{\pi^2}
					+ \sum_{w \in C_n} G(x,w)
					\,  G_n(0,w). \qedhere
				\end{eqnarray*}
			\end{proof}

			\section{Slowly recurrent set and  the subsequential limit}\label{sec:slow}
			Simple random walk paths in $\mathbb{Z}^4$ are ``slowly recurrent'' sets in
			the terminology of \cite{LSlowly}. In Section~\ref{nicesec} we will consider a
			subcollections $\niceset_n$ of the collection of slowly recurrent sets and
			give uniform bounds for escape probabilities for such sets. Then in Section~\ref{subseqsec} we use these estimates to prove  Proposition~\ref{nov22.cor5} and~\ref{subsequenceprop}.
			{This section does  not rely on notions and results in  \cite{LSlowly,3rdmoment} as we  will give a new and self-contained treatment.  
			}
			
			\subsection{Sets in \texorpdfstring{$\niceset_n$}{Lg}}
			\label{nicesec}
			Given a subset $V \subset \mathbb{Z}^4$ and $m\in \mathbb{N}$ we write 
			\begin{equation*}
				V_m = V \cap A(m-1,m), \qquad \textrm{and}\qquad  h_m = h_{m,V} = H(V_m).
			\end{equation*}
			Using \greg{\eqref{nov15.1.improve}},  we can see that there exist $0 < c_1 < c_2 <
			\infty$ such that 
			\begin{equation}\label{eq:Harnack}
				c_1 \, h_m \leq H(z,V_m) \leq c_2 \, h_m , \;\;\;\;\; \forall \,z \in
				C_{m-2} \cup A(m+1,m+2).
			\end{equation}
			
			\begin{definition}\label{def:slow}
				Let $\niceset_n$ denote the collection of subsets $V$ of $\mathbb{Z}^4$ such
				that for all integers $m \geq \sqrt n$, 
				\begin{equation*}
					H(V_m) \leq \frac{\log^2m}{m} .
				\end{equation*}
			\end{definition}
			
			Note that  $\niceset_1 \subset \niceset_2 \subset \cdots$.
			Also if $\tilde V \subset V$ and $V \in \niceset_n$,
			then $\tilde V \in \niceset_n$,   The following is
			an immediate corollary of Proposition \ref{nov15.prop1}.
			
			\begin{proposition}
				$\Prob\{S[0,\infty) \not\in \niceset_n\} $ is a fast decaying sequence.  
			\end{proposition}
			
			Let $E_m$ denote the event 
			\begin{equation*}
				E_m = E_{m,V}= \{S[1,\sigma_m] \cap V = \eset\}.
			\end{equation*}
			{Note  that $\Prob(E_m) = \Es(V;m)$. We will interchangeably use the  two notions $\Prob(E_m)$ and $\Es(V;m)$ throughout this section.}   
			We write $\hm_m(z)$ for the harmonic measure of $\p C_m$ for random walk
			starting at the origin, that is,
			\begin{equation*}
				\hm_m(z) = \Prob\{S_{\sigma_{m}} = z\}, \qquad \forall  z \in \p C_m .
			\end{equation*}
			If $V \subset \mathbb{Z}^4$ and $\Prob(E_m) > 0$, we write 
			\begin{equation*}
				\hm_m(z;V) = \Prob\{S_{\sigma_{m}}  = z\mid E_m\} .
			\end{equation*}
			By the strong Markov property, we have
			\[
			\Prob\{S[\sigma_m,\sigma_{m+1}] \cap V \neq \eset\} = \sum_{z \in \p C_m} %
			\hm_m(z) \, \Prob^z\{ S[0,\sigma_{m+1}] \cap V \neq \eset\},\]  
			and
			\begin{eqnarray}
				\Prob(E_{m+1}^c \mid E_m) & = & \Prob\{S[\sigma_m,\sigma_{m+1}] \cap V \neq %
				\eset \mid E_m \} \label{eq:Markov2} \\
				& = &\sum_{z \in \p C_m} \hm_m(z;V) \, \Prob^z\{ S[0,\sigma_{m+1}] \cap V
				\neq \eset\}.\nonumber
			\end{eqnarray}
			
			\begin{proposition}
				\label{prop:3.11} There exists $c < \infty$ such that if $V \in \niceset_n$, 
				$m \geq n/10$, and $\Prob(E_{m+1} \mid E_{m}) \geq 1/2$, then 
				\(\Prob(E_{m+2}^c \mid E_{m+1})\leq c \, \log^2 n/n.\)
			\end{proposition}
			\begin{proof}
				As in \eqref{eq:Markov2}, we write
				\[    \Prob(E_{m+2}^c \mid E_{m+1})
				= \sum_{z \in \partial
					C_{m+1}} \hm_{m+1}(z;V) \, \Prob^z\{S[0,\sigma_{m+2}]
				\cap V \neq \emptyset\}.\]
				Using $\Prob(E_{m+1} \mid E_{m})
				\geq 1/2$, we claim that there exists $c<\infty$ such that  
				\begin{equation}\label{eq:escape}
					\hm_{m+1}(z;V) \leq c \, \hm_{m+1}(z), \qquad \forall\; z\in \partial C_{m+1}.
				\end{equation}
				Indeed,
				we have 
				\begin{eqnarray*}
					\hm_{m+1}(z;V) & = & \frac{\Prob\{S_{\sigma_{m+1}}  = z , E_{m+1}\}}
					{\Prob(E_{m+1})} \\
					&  \leq & 2 \,  \frac{\Prob\{S_{\sigma_{m+1}} = z , E_{m+1}\}}
					{\Prob(E_{m})} \leq 2 \, \Prob\{S_{\sigma_{m+1}} = z \mid E_{m}\},
				\end{eqnarray*}
				and 
				the Harnack inequality
				shows that 
				\[   \Prob\{S_{\sigma_{m+1}} = z \mid E_{m}\}
				\leq \sup_{w \in \p C_m} \Prob^w\{S_{\sigma_{m+1}} = z\}
				\leq c \, \hm_{m+1}(z).\]
				
				Therefore, letting $r_k=\Prob\{S[\sigma_{m+1},\sigma_{m+2}]
				\cap V_k \neq \eset \}$ for $k\in \N$, we have 
				\begin{eqnarray*}
					\Prob(E_{m+2}^c \mid E_{m+1})
					& \leq &   c\,\sum_{z \in \partial C_{m+1}} \hm_{m+1}(z) \, \Prob^z\{S[0,\sigma_{m+2}]
					\cap V \neq \emptyset\} \\
					& =&  c \, \Prob\{S[\sigma_{m+1},\sigma_{m+2}]
					\cap V \neq \eset \}\leq c\, \sum_{k=1}^{m+2} r_k.
				\end{eqnarray*}
				By Definition~\ref{def:slow}, the terms $r_k$ for $k=m,m+1,m+2$ are bounded by
				\begin{eqnarray*}   \Prob\{S[\sigma_{m+1},\sigma_{m+2}]
					\cap (V_m \cup V_{m+1} \cup V_{m+2}) \neq \eset \}
					& \leq & H(V_m \cup V_{m+1} \cup V_{m+2})\\
					&  \leq &
					\frac{c\, \log^2 n}{n}.
				\end{eqnarray*}
				Using \eqref{nov15.1}, we see that for $\lambda$ large enough,
				\[  \sum_{k=1}^{m-\lambda\log m}r_k\le 
				\Prob\{S[\sigma_{m+1},\sigma_{m+2}]
				\cap C_{m- \lambda \, \log m} \neq \eset \}
				\leq  c \,n^{-2}.\]
				For $m-\lambda \log m \leq k \leq m-1$, 
				\eqref{nov15.1.improve} and the definition of $\niceset_n$
				imply that
				\[ r_k\leq c \, e^{-2(m-k)} \, H(V_k) \le  c \, e^{-2(m-k)} \, \frac{\log^2k}{k}.\]
				Summing over $k$ gives the result.
			\end{proof}
			
			\begin{definition}\label{def-slow-set}
				Let $\nicesettwo_n$ denote the set of $V \subset \niceset_n$ such that $\Prob%
				(E_n) \geq 2^{-n/4}$.
			\end{definition}
			
			The particular choice of $2^{-n/4}$ in this Definition~\ref{def-slow-set} is rather
			arbitrary but it is convenient to choose a particular fast decaying
			sequence. For typical sets in $\niceset_n$ one expects that $\Prob(E_n)$
			decays as a power in $n$, so ``most'' sets in $\niceset_n$ with $\Prob(E_n)
			> 0$ will also be in $\nicesettwo_n$.
			
			{
				Recall $\Es(V;n) = \Prob\{\rev{S[1,\sigma_n]} \cap 
				V = \eset\}$, which is decreasing in $n$. 
				We state the next immedate fact as a proposition
				so that we can refer to it.

				\begin{proposition}\label{lem:nice}
					For any $r>0$ and any random subset $V\subset \Z^4$,
					\[
					\E\left[\Es(V;m)^r;{V\notin \nicesettwo_n}\right]\le \P[ V\notin \niceset_n]+2^{-rn/4}.
					\]
					In particular, if $\P[ V\notin \niceset_n]$ is fast decaying
					then so is the left-hand side. 
					
				\end{proposition}
			}
			
			\begin{proposition}
				\label{dec4.prop1} There exists $c < \infty$ such that if $V \in \nicesettwo%
				_n$, then 
				\begin{equation*}
					\Prob(E_{j+1}^c \mid E_{j})\leq \frac{c \, \log^2 n}{n}, \qquad\frac{3n}{4%
					} \leq j \leq n .
				\end{equation*}
			\end{proposition}
			\begin{proof}
				If $\Prob(E_{m+1} \mid E_m) <1/2$ for all $n/4 \leq m
				\leq n/2$, then $\Prob(E_n) <2^{-n/4}$ and
				$V \not\in \nicesettwo_n$.
				Therefore we must have $\Prob(E_{m+1} \mid E_m) \geq 1/2$ for
				some $n/4 \leq m \leq n/2$. Now  (for $n$ sufficiently
				large) we can use Proposition~\ref{prop:3.11} and
				induction to conclude
				that $\Prob(E_{k+1} \mid E_k) \geq 1/2$ for $m \leq k 
				\leq n$.  The result then follows from Proposition~\ref{prop:3.11}.
			\end{proof}

			It follows from Proposition~\ref{dec4.prop1} that there exists $n_0$  such that $\nicesettwo_n \subset \nicesettwo_{n+1}$ for $n \geq n_0$. We
			fix the smallest such $n_0$ and set 
			\begin{equation*}
				\nicesettwo = \bigcup_{j=n_0}^\infty \nicesettwo_j .
			\end{equation*}
			Combining Proposition~\ref{prop:3.11} and~\ref{dec4.prop1},  and the union bound, we have  
			\begin{equation}  \label{jun6.4}
				\Prob(E_{n+k}^c \mid E_n) \leq \frac{c \, k\, \log^2 n}{n}, \qquad  k\in \N, \;V \in %
				\nicesettwo_n.
			\end{equation}

			\begin{proposition}  There exists $c < \infty$
				such that if $V \subset C_n$ and $V  \in \nicesettwo_n$, then
				\begin{equation}  \label{jul19.2}
					\Es[V;n]
					\,\left[1-  \frac{c\, \log^2 n}{n}
					\right]  \leq \Es[V] \leq 
					\Es[V;n].
				\end{equation}
			\end{proposition}
			
			\begin{proof}  The upper bound is trivial.  For
				the lower bound, we first use the previous proposition 
				to see that
				\[  \Es(V;n+1) \geq 
				\Es(V;n)
				\, \left[1  - O(\log^2n/n)\right].\]
				Since $\Es(V)=\Es(V;n+1)(1-\max_{z\in \partial C_{n+1}} H(z,V))$, it suffices to show that there exists
				$c$ such that for all $z \in \partial C_{n+1}$,
				\[     H(z,V) \leq c \, \log^2 n/n.\]
				This can be done by dividing $V$ into $V_j$'s  similarly as  in the proof of 
				Proposition \ref{prop:3.11}, (see the bound for $\sum_1^{m+2}r_j$ there).
			\end{proof}
			
			The next proposition is the key to the analysis of slowly recurrent sets. It
			says that the distribution of the first visit to $\p C_n$ given that one has
			avoided the set $V$ is very close to the unconditioned distribution. We
			would not expect this to be true for recurrent sets that are not slowly
			recurrent.
			
			\begin{proposition} \label{cond}
				There exists $c < \infty$ such that if $V \in \nicesettwo_n$ we have 
				\begin{equation}  \label{jun6.2}
					\hm_n(z;V) \leq \hm_n(z) \, \left[ 1 + \frac{c \, \log^3 n}{n}\right], \qquad \forall \, z \in
					\partial C_n.
				\end{equation}
				Moreover, 
				\begin{equation}  \label{jun6.1}
					\sum_{z \in \partial C_n} |\hm_n(z) -\hm_n(z;V)| \leq \frac{c \, \log^3 n}{n}%
					.
				\end{equation}
			\end{proposition}
			
			\begin{proof}  Let $k = \lfloor \log n \rfloor $.  
				By \eqref{jun6.4}, we have 
				
				\begin{equation}
					\label{jun3.1}     \Prob(E_{n}^c \mid E_{n-k})
					\leq  \frac{c \, \log^3 n}{n}.
				\end{equation}
				Consider a random walk starting   on $ \p
				C_{n-k}$ with the distribution $\hm_{n-k}
				(\cdot;V) $  and let $\nu$ denote
				the distribution of the first visit
				to $\partial C_n$.  In other words, $\nu$ is the
				distribution of the first visit to $\partial
				C_n$ conditioned on the event $E_{n-k}$.
				Using \eqref{harmonic}, we see that
				for $z \in \partial C_n$,
				\[          \nu(z) = \hm_n(z) \, [1 + O(n^{-1})].\]
				By \eqref{jun3.1}, for each $z \in \partial C_n$, we have
				\begin{align*}
					\hm_n(z;V)& =  \Prob\left\{S_{\sigma_n}  =   z \mid E_{n}\right\}  \leq \frac{\Prob(E_{n-k})}{\Prob(E_n)} \, 
					\Prob\left\{S_{\sigma_n}= z \mid E_{n-k}\right\}
					\\ & \leq \frac{\nu(z)}{1-\P(E^c_n|E_{n-k})}  \leq   \hm_n(z) \, \left[1 + O\left(\frac{\log^3n}n \right)
					\right].
				\end{align*}
				Since $\hm_n(\cdot)$ and $\hm_n(\cdot;V)$ are probability measures
				on $\partial C_n$, we have 
				\begin{align*}
					\sum_{z \in \partial C_n}  |\hm_n(z) - \hm_n(z;V) |
					& =  2 
					\sum_{z \in \partial C_n}  [ \hm_n(z;V) - \hm_n(z) ]_+\\
					& \leq \frac{c \, \log^3 n}{n} \sum_{z \in \partial C_n}
					\hm_n(z)  \leq  \frac{c \, \log^3 n}{n}.\qedhere
				\end{align*}
			\end{proof}

			\subsection{Along a subsequence}
			\label{subseqsec}
			In this section we prove Proposition~\ref{nov22.cor5} and~\ref{subsequenceprop}
			via the estimates proved for sets in $\nicesettwo$.
			For $V\in \nicesettwo_{n^{4}}$, let
			\[
			V_{n}^{\ast }:=V\cap \left\{ e^{(n-1)^{4}+4(n-1)}\leq |z|<
			e^{n^{4}-4n}\right\}.
			\]
			\begin{proposition}
				\label{prop.jun6} There exists $c<\infty $ such that 
				{if $V\in \nicesettwo_{n^{4}}$,
					\begin{equation*}
						\Es(V;(n+1)^4)
						= \Es(V;n^4) \,\left[ 1-H(V_{n+1}^{\ast })+O\left( 
						\frac{\log ^{2}n}{n^{3}}\right) \right] .
					\end{equation*}}
				\end{proposition}
				
				
				\begin{proof} 
					Let $\tau_n =\inf\{j: S_j\notin C_{n^4} \}$ and  recall $E_{n} $ and $\hm_n(z;V)$. 
					We observe that $(\Es(V;n)-\Es(V;(n+1)^4))/\Es(V;n^4)$ is bounded by 
					\[
					\P\{S[\tau_n,\tau_{n+1}] \cap V^*_{n+1}\neq \emptyset|E_{n^4}\}+\Prob\{S[\tau_n,\tau_{n+1}] \cap (V \setminus V^*_{n+1})\neq \emptyset|E_{n^4} \}.
					\]	
					To bound the first term, we have 
					\begin{align*}
						&\left |\Prob\{S[\tau_n,\tau_{n+1}] \cap V^*_{n+1}\neq \eset \mid E_{n^4}\}-\Prob\{S[\tau_n,\tau_{n+1}] \cap V^*_{n+1}
						\neq \eset\} \right |\\
						\le&\sum_{z \in \p C_{n^4}}| \hm_{n^4}(z;V)-\hm_{n^4}(z)|  \, \Prob^z\{ S[0,\tau_{n+1}] \cap V^*_{n+1}
						\neq \eset\}\\
						\le & \;\frac{  c\log^3 n}{n^4}\,\sup_{z\in\p C_{n^4} } \Prob^z\{ S[0,\tau_{n+1}] \cap V^*_{n+1} \neq \eset\}   \\
						\le  &  \;\frac{  c\log^3 n}{n^4}\,\Prob\{ S[\tau_n,\tau_{n+1}] \cap V^*_{n+1} \neq \eset\},
					\end{align*}
					where the three inequalities are due to the strong Markov property, \eqref{jun6.2} and Harnack inequality respectively.
					
					Using \eqref{harmonic}, we have \( \Prob\{S[\tau_n,\tau_{n+1}] \cap V^*_{n+1}
					\neq \eset\} = H(V^*_{n+1})\, {[1+O(e^{-4n})]}.\)
					Hence, it suffices to prove that
					\[ \Prob\{S[\tau_n,\tau_{n+1}] \cap (V \setminus V^*_{n+1})
					\neq \eset  \mid E_{n^4}\} = O\left( \frac{  \log^2 n}{n^3}\right),\]
					which by the strong Markov property and  \eqref{jun6.2}, can be further reduced to showing that
					\[  \Prob\{S[\tau_n,\tau_{n+1}] \cap (V \setminus V^*_{n+1})
					\neq \eset  \} = O\left( \frac{  \log^2 n}{n^3}\right).\]
					Note that $V \setminus V^*_{n+1}$ is contained in the
					union of $C_{n^4-4n}$ and    $O(n)$  sets of the form $V_m$
					with $m \geq n^4 - 4n$.
					By Definition~\ref{def:slow} and the union bound,
					\[  H\left((V \setminus V^*_{n+1} )
					\cap \left\{e^{n^4 - 4n} \leq |z| \leq e^{(n+1)^4} \right\}\right)
					= O\left(\frac{\log^2 n}{n^3}\right). \]
					By Lemma \ref{standardlemma},  \( \Prob\left\{S[\tau_n , \tau_{n+1}] \cap   C_{n^4 - 4n}\neq \eset \right\} =  O(e^{-8n})$.
					This concludes the proof. \end{proof}

					\begin{corollary}
					\label{cor.jun22} If $V \in \nicesettwo_{n^4}$, $m \geq n$, and $m^4 \leq k
					\le (m+1)^4$, then 
					\begin{equation*}
					\Es(V;k)  = \Es(V;n^4)\, \exp\left\{- \sum_{j=n+1}^m H(V_j^*)
					\right\} \, \left[ 1 + O\left(\frac{\log^4 n}{n } \right) \right].
					\end{equation*}
					\end{corollary}
					
					\begin{proof}
					
					By Proposition~\ref{prop.jun6}, if $m>n$ we have 
					\begin{eqnarray*}
					\frac{\Prob(E_{m^4})}{\Prob(E_{n^4})}
					& = &  
					\, \prod_{j={n+1}}^m \left[1 - H(V_j^*) 
					+ O\left(\frac{\log^2 j}{j^3}\right) \right].
					\end{eqnarray*}
					By Definition~\ref{def:slow} and the union bound,  we have  $H(V_j^*) = O(\log^2 j/j)$. Hence 
					\begin{eqnarray*}
					\frac{\Prob(E_{m^4})}{\Prob(E_{n^4})}
					& = &  
					\prod_{j={n+1}}^m  \left[ e^{-H(V_j^*)}  + O
					\left(\frac{\log^4 j}{j^2}\right) \right]
					\\ & = & \left[1 + O\left(\frac{\log^4 n}{n}\right) \right]
					\,\exp\left\{-\sum_{j=n+1}^m   H(V_j^*)  \right\}.
					\\
					\end{eqnarray*}
					On the other hand, \eqref{jun6.4} implies that  \(  \Prob(E_k) = \Prob(E_{m^4}) \, \left[1 -
					O \left(\log^2 m/m \right) \right] \) for $m^4 \leq k \leq (m+1)^4$. This concludes the proof.
					\end{proof}
					Now we apply our theory to LERW.  Recall the setup in Section~\ref{subsec:outline}. 
					\begin{proof}[Proof of Proposition~\ref{nov22.cor5}]  
					Let $k =\lceil \log^2 n \rceil$.  We will show the
					stronger result that
					\begin{equation}\label{eq:pn}
					p_n = \E\left[\Es(\Gamma_n)^r\right] = p_{n-k} \, \left[ 1 + O(\log^4 n/n)\right], 
					\end{equation}
					ad similarly for $p_{n+1}$.   Since  $p_n$ decays polynomially, by  Proposition~\ref{lem:nice},
					\begin{equation}\label{eq:pn-nice}
					\E\left[\Es(\Gamma_n)^r\right] = \E\left[\Es(\Gamma_n)^r1_{\Gamma_n\in \nicesettwo_n}\right] (1+\eps_n),
					\end{equation}where $\eps_n$ fast decaying.
					By \eqref{jul19.2} 
					we have
					\begin{equation}\label{eq:pn-trunk}
					\E[\Es(\Gamma_n)^r]= \E[\Es(\Gamma_n;n)^r] \, [1 + O(\log^2n/n)].
					\end{equation}
					By Lemma \ref{nov21.lemma1}, except for 
					an event of fast decaying probability,
					\begin{equation}\label{eq:good-event}
					\Gamma \cap C_{n-k}
					\subset \Gamma_n
					\subset V
					\end{equation}where $V:=(\Gamma \cap
					C_{n-k}) \cup  \left[S[0,\infty) \setminus C_{n-k}\right].$
					If \eqref{eq:good-event} occurs, then  we have 
					\[     \Es[\Gamma; n - k]=\Es[V;n-k] 
					\geq \Es[\Gamma_n;n] \geq \Es[\Gamma_n;n+k]  \geq \Es
					[V;n+k]\]
					{By  \eqref{jun6.4},  we have 
						\[   \E\left[\Es[V;n +k]^r ;{V\in \nicesettwo_{n-k}} \right] 
						\hspace{1.5in}\]
						\[ = \E[\Es[V;n-k]^r;{V\in \nicesettwo_{n-k}}]
						\,  \left [1 - O\left({\log^4 n}/{n}
						\right)\right], \]
						Since $\E\left[\Es[V;n +k]^r\right]$ decays like a power of $n$, by  Proposition~\ref{lem:nice} },
					
					\[   \E\left[\Es[V;n +k]^r \right] = \E[\Es[V;n-k]^r]
					\,  \left [1 - O\left({\log^4 n}/{n}
					\right)\right]. \]
					Now \eqref{eq:pn} follows from \eqref{eq:pn-trunk} and 
					\[ \E[\Es(\Gamma_n;n)^r]\ge   \E\left[\Es[V;n +k]^r \right] = 
					p_{n-k} \,  \left [1 - O\left({\log^4 n}/{n}
					\right)\right].\]
					A similar argument gives $p_{n+1}=p_{n-k} \,  \left [1 - O\left({\log^4 n}/{n}\right)\right]$.
					\end{proof}

					A useful corollary  of the proof is
					\begin{equation}  \label{jul19.4}
					\E[\Es(\Gamma;n)^r] = p_n
					\, \left[1 + O(\log^4n/n)\right].   
					\end{equation}

					\begin{proof}[Proof of Proposition~\ref{subsequenceprop} ]
					Let $Q_n = \Es[\Gamma;n^4]$ and 
					\[
					\Gamma_n^* = \Gamma \cap A((n-1)^4 + 4(n-1), n^4 - 4n) .
					\] 	
					Then, by Proposotion~\ref{prop.jun6}, if $\Gamma \in \nicesettwo_{n^4}$, we have 
					\[          Q_{n+1} =   Q_n \, \left[1  - H(\Gamma_{n+1}^*) +
					O\left( \frac{\log^2 n}{n^3}\right) \right] .\]	
					
				{Applying Proposition~\ref{lem:nice} to $V=\Gamma$, we have} 
				\[
				\E\left [  Q_{n+1}^r\right]= \E\left [ Q_n^r \, [1  - r\, H( \Gamma_{n+1}^*)]\right] + \E\left[  Q_n^r\right] \, O\left( \frac{\log^2 n}{n^3}\right).
				\]
				Recall $\wt I(\cdot, \cdot)$ in \eqref{jun7.1}. 
				We see that $\Prob\{\Gamma_{n+1}^* \neq \tilde \Gamma_{n+1}\mid \F_{n^4}\} $ is bounded by
				\begin{align*}
				\P\{\wt I[n^4+n, n^4+4n] |\F_{n^4} \}  +\P\{\wt I[(n+1)^4-4(n+1), (n+1)^4-(n+1)] |\F_{n^4} \},
				\end{align*}
				which by \eqref{jun7.1} is further bounded by  $O(n^{-4})$. Therefore, 
				\[  \E\left[Q_n^r \,|H(\Gamma_{n+1}^*) -  H(\tilde \Gamma_{n+1})|
				\right] \leq  O(n^{-4}) \, \E\left[Q_n^r\right].
				\] 
				Hence,\begin{equation}\label{eq:replace}
				\E\left [ Q_{n+1}^r\right]
				= \E\left [ Q_n^r \, [1  - r\, H(\tilde \Gamma_{n+1})]\right 
				] + \E\left[  Q_n^r\right] \, O\left( \frac{\log^2 n}{n^3}\right).
				\end{equation}
				Using \eqref{harmonic} in Lemma \ref{standardlemma}, we can see that
				\begin{equation}\label{eq:harnack1}
				\E[H(\tilde \Gamma_{n+1}) \mid \F_{n^4}]=\E[H(\tilde \Gamma_{n+1})] \, [1 + o(e^{-4n})] 
				=\tilde h_{n+1}\, [1 + o(e^{-4n})  ]. 
				\end{equation}
				Let  $q_n:=p_{n^4}$. Combining \eqref {jul19.4}, \eqref{eq:replace} and~\eqref{eq:harnack1}, we get
				\begin{eqnarray*}
				q_{n+1} & = & q_n \, \left[1 - r\,  \tilde h_{n+1} +
				O\left( \frac{\log^2 n}{n^3}\right) \right]\\
				& = & q_n\,\exp\left\{-r\, \tilde h_{n+1}\right\}
				\left[1 +O\left(\frac{1}{n^2} \right) \right],
				\end{eqnarray*}
				where the last inequality uses $\tilde h_{n+1} = O(n^{-1})$.
				In particular, if $m > n$, 
				\[  q_m  =  q_n \,  \left[1 +O\left(\frac{1}{n} \right)
				\right]\exp\left\{-r\sum_{j=n+1}^m \tilde h_j\right\}  .\]	
				from which Proposition~\ref{subsequenceprop}  follows.
				\end{proof}
				By Proposition~\ref{nov22.cor5}, we see that 
				\begin{equation*}
				p_m = p_{n^4} \, \left[1 + O\left({\log^4 n/n} \right) \right] \qquad \textrm{if }n^4 \leq m \leq (n+1)^4.
				\end{equation*}
				Combined with Proposition~\ref{subsequenceprop}, we immediately get 
				\begin{corollary}
				\label{coro:sub} There exists $c_0 < \infty$ such that as $m \rightarrow
				\infty$, 
				\begin{equation*}
				p_m = \left[c_0 + O\left(\frac{\log^4 m }{m^{1/4}}\right) \right] \,
				\exp\left\{-r\sum_{j=1}^{\lfloor m^{1/4} \rfloor } \tilde h_j\right\}.
				\end{equation*}
				\end{corollary}
				
				
				\section{From the subsequential limit to the full limit}\label{sec:full}
				In  this section we prove Proposition~\ref{nov22.prop3} and~\ref{prop:rs}.
				The proof  of  Proposition~\ref{nov22.prop3} is the technical bulk of this section, which will be given in Section~\ref{comparesec}.
				Let us first conclude the proof of Proposition~\ref{prop:rs} assuming Proposition~\ref{nov22.prop3}.
				\begin{proof}[Proof of Proposition~\ref{prop:rs}]
				Given Proposition~\ref{nov22.prop3}, the $s=0$ case follows from Corollary~\ref{coro:sub}.
				For the general case, recall that \xin{$1\le G_n\le 8$} and $G_n$ converge to $G_\infty$ almost surely. 	Moreover Lemma \ref{nov21.lemma1} implies
				that there exists a fast decaying sequence $\{\epsilon_n\}$
				such that if $m \geq n$, $\P\{|G_n - G_m| \geq \epsilon_n\} \leq \epsilon_n$.
				Therefore \( \E \left[  \phi_n^{-r}
				\, Z_n^r \, G_n^{-s}\right] - \E\left[  \phi_n^{-r}\, Z_n^r \, G_\infty
				^{-s} \right]\) is fast decaying and
				\[     \left|\E\left(\phi^{-r}_n \, Z_n^r \, G_\infty^{-s}\right) 
				- \E\left(\phi_m^{-r}
				Z_m^r \, G_\infty^{-s}\right)  \right| \leq c
				\,   \left|\E\left(  \phi_n^{-r} \, Z_n^r \right) 
				- \E\left(  \phi_m^{-r}
				Z_m^r \right)  \right|.\]
				Take $m\to \infty$, we see that the $s\neq 0$ case follows from the $s=0$ case.
				\end{proof}

				\subsection{Harmonic measure of the range of SRW}\label{sec:deviation}
				We start by proving two  estimates for the harmonic measure of the range of random walk.

				\begin{lemma}
				\label{nov18.lemma1} Let 
				\begin{equation*}
				\sigma_n^- = \sigma_n - \lfloor n^{-1/4} \, e^{2n} \rfloor, \;\;\;\;
				\sigma_n^+ = \sigma_n + \lfloor n^{-1/4} \, e^{2n} \rfloor,
				\end{equation*}
				\begin{equation*}
				n' = \lceil n + n^{4/5} \rceil, \;\;\;\;	\smallbit_n^- = S[0,{\sigma_n^-}],\;\;\;\; \smallbit_n^+ =S[\sigma_n^+,\sigma_{n'}],
				\end{equation*}
				\begin{equation*}
				R_n = \max_{x \in \smallbit_n^-} H(x,\smallbit^+_n) + \max_{y \in %
					\smallbit_n^+} H(y, \smallbit_n^-),
				\end{equation*}
				Then, for all $n$ sufficiently large, 
				\begin{equation}  \label{dec1.5}
				\Prob\{R_n \geq n^{-1/6} \} \leq n^{-1/3} .
				\end{equation}
				\end{lemma}
				Our proof will actually give a stronger estimate, but \eqref{dec1.5} is all
				that we need and makes for a somewhat cleaner statement.

				\begin{proof}   
				{
					Let $m = n+ \lceil  \log n \rceil$.  
					Recall from \eqref{jul18.1} that there exists
					$c_0 < \infty$ such that 
					\[
					\Prob\{\sigma_{n} \geq c_0 \, e^{2n} \,  \log n \} \leq n^{-1} \qquad\textrm{and}\qquad \Prob\{\sigma_{m} \geq c_0 \, e^{2n} \,  n^2
					\, \log n  \} \leq n^{-1}.
					\]Let $V = V_n = 
					S[\sigma_n^+,\sigma_m]$,
					\[   \tilde  R_{n} = 
					\max_{x \in \smallbit_n^-} H(x, V) + \max_{y \in 
						V} H(y, \smallbit^-_n),\]
					\[    R_n^* = 
					\max_{x \in \smallbit_n^-} H(x, \smallbit^+_n
					\setminus V) + \max_{y \in 
						\smallbit_n^+ \setminus V} H(y, \smallbit_n^-),\]
					and note that $R_n  \leq  \tilde R_n + R ^*_n$.
					We claim that for $n$ sufficiently large,
					\begin{equation}  \label{jul20.1}
					\Prob\{R^*_n \geq n^{-1/6}/2\} \leq  O(n^{-1}).
					\end{equation}
					To see this,  let $\ell = \lfloor n +[ (\log n)/2] \rfloor,$
					and let $U$ denote the event 
					$U = \{(S_n^+ \setminus V) \cap C_{ \ell}
					= \eset \}.$  By \eqref{nov15.1}, $\Prob(U) \geq
					1 - O(n^{-1})$,  and on the event
					$U$ we have $\smallbit_n^- \subset C_n
					\subset C_{\ell} \subset (S_n^+ \setminus
					V)^c$.  
					Therefore, on the event $U \cap
					\{S[0,\infty) \in \niceset_n\}$, 
					for $y \in \smallbit_n^-$, $x \in S_n^+ \setminus V$, we have the following two bouds:
					\begin{align*}
					H(y,S_n^+ \setminus V ) &\leq c \, H( S[0,\infty) \cap A(n+1,n + n^{4/5})\} \leq c \, n^{-1/5} \, \log^2 n;\\ 
					H(x,\smallbit_n^-) &\leq H(x,C_n) \leq c/n.
					\end{align*}
					This gives \eqref{jul20.1}.

					Let $N = N_n =  \lceil  c_0 \, e^{2n} \, \log n
					\rceil, M = M_n =\lceil  c_0 \, e^{2n} \, n^2\, \log n
					\rceil $  and $k = k_n = \lfloor n^{-1/4} \, e^{2n}/4\rfloor$.  		
					For each integer $0\le j \le N/(k+1)$, let
					$E_j = E_{j,n}$ be event that at least one of the following holds:
					\begin{equation}\label{eq:max}
					\max_{0\le i \leq jk} H(S_i, S[(j+1)k,M]) \ge \frac{ \log n}{n^{1/4}},
					\end{equation}
					\begin{equation}\label{eq:maxtwo}
					\max_{(j+1)k \leq i \leq M} H(S_i, S[0,jk]) \ge \frac{ \log n}{n^{1/4}}.
					\end{equation}
					Then for large enough $n$, we have 
					$\{\tilde R_n\ge n^{-1/3},\;  jk \le \sigma_n\le (j+1)k\}\subset  E_j.$
					
					Recall the notion $Y_{n,\alpha}$ in Corollary~\ref{cor:har}. For fixed $j$, using
					the reversibility of simple random walk, the probabilities of both the event in  \eqref{eq:max} and in \eqref{eq:maxtwo} are bounded by 
					$\P [Y_{M,N/k}  \ge n^{-1/4} \log n]= O(n^{-3/4})$. 
					Therefore $\P (E_j) \le  O(n^{-3/4})$ and hence
					\[    \Prob  \{ \tilde R_n\ge n^{-1/3}\}\le O(n^{-1}) +\frac{N}{k}  O(n^{-3/4})\le O\left(\frac{\log n}{n^{1/2} } \right).\]
					Combining this with \eqref{jul20.1} gives the proof.}
				\end{proof}

				%
				%
				%
				%
				
				The next lemma will use the notion of capacity $\Cp(V)$ for  a subset $V \subset \mathbb{%
					Z}^4$. We will not review the definition but only recall three  key facts:
				\begin{align}
				&\Cp(V\cup V')\le \Cp(V) +\Cp(V') \qquad &&V, V'\subset \Z^4;\label{eq:subadd}\\
				& \Cp( \{z+v: v\in V \}   )=\Cp(V)  \qquad&& V\subset \Z^4, \; z\in \Z^4;\label{eq:shift}\\
				& \Cp(V)   \asymp |z|^2 H(z,V)\qquad &&  V \subset C_n,\;z \not\in C_{n+1} \label{jun9.1}.
				\end{align}
				See   \cite[Section 6.5, in particular, Proposition 6.5.1]{lawler-limic} for definitions and
				properties. {Combining these
					with the estimates for hitting probabilities, we have 
					\[
					\E[\Cp \left(S[0,n^2]\right)] \asymp n^2/\log n.
					\]
					By Markov inequality, there exists $\delta > 0$ such that
					\begin{equation}\label{eq:cap}
					\Prob\left \{\Cp \left(S[0,n^2]\right)
					\leq \frac{  n^2}{\delta \, \log n} \right\} \geq \delta.
					\end{equation}
					By  iterating \eqref{eq:cap} and using the strong Markov property and subadditivity as in the proof of
					Proposition \ref{nov15.prop1}, we see that there exists $c,\beta$
					such that for all $n$ and all $a > 0$,
					\begin{equation}  \label{capest}
					\Prob\left \{\Cp \left(S[0,n^2]\right) \geq \frac{a n^2}{\log n} \right\} \leq c\, e^{-\beta a}.
					\end{equation}
				}
				
				\begin{lemma}
				\label{jun7.lemma} For all $j,m\in \N$, let \(L[j,m]=\Cp\left( S[j,j+m]\right)\). For $k,n\in \N$, let \(\bar{L}
				(n;k)=\max_{j\leq n}L[j,k]\).
				Then for every $u<\infty $ 
				\begin{equation}
				\Prob\left\{ \bar{L}(n^{u}\,e^{2n};n^{-1/4}\,e^{2n})\geq
				2\,n^{-11/10}\,e^{2n}\right\} \qquad\textrm{is fast decaying}.  \label{jun7.6}
				\end{equation}%
				
				\end{lemma}
				
				\begin{proof}
				Let $k = \lceil n^{-1/4} \, e^{2n} \rceil .$ 
				Let $U$ denote the 
				event in \eqref{jun7.6}.
				By the subadditivity of capacity \eqref{eq:subadd} and the union bound, we have
				\[   U \subset   \bigcup_{i=1}^{n^{u+1}}  
				\left\{L[ik, k] \geq n^{-11/10}\, e^{2n}  \right\}. \]
				By \eqref{eq:shift}, the events in the union
				are identically distributed. {Therefore
					\[   \Prob(U) \leq 
					n^{u+1} \,\Prob\left\{ L[0;k] \geq n^{3/20} \, \frac{e^{2n}}{n^{1/4}
						\, n } \right\}, \]
					which is fast decaying by \eqref{capest}.}
				%
				%
				\end{proof}
				\subsection{Proof of Proposition~\ref{nov22.prop3}}
				\label{comparesec}
				The  strategy is to find a $u > 0$, and for
				each $n$ 
				a  random set  $U = U(n)
				\subset \Z^4$ that can be written as a disjoint union 
				\begin{equation}\label{eq:U}
				U =  \bigcup_{j=(n-1)^4+ 1}^{n^4} U_j
				\end{equation}
				such that the following four conditions hold where
				\begin{align}
				&U \subset \tilde \Gamma_n;\label {nov18.1} \\
				&	U_j \subset \eta^j, \qquad  j=(n-1)^4+1,\ldots n^4;  \label {nov18.2}\\
				&\E\left[ H(\tilde \Gamma_n \setminus U)\right]+
				\sum_{(n-1)^4<j\le n^4} \E\left[ H ( \eta^j \setminus U_j)\right] \leq O(n^{-(1+u)});	 \label {nov18.3} \\
				&		\max_{{(n-1)^4 < j   \leq n^4}}\;  \max_{x \in U_j}\; 
				H(x,U \setminus U_j) \leq  n^{-u} .  \label {nov18.4}
				\end{align}
				We will first show  that finding such a set gives
				the result.   
				Taking expectations and using \eqref{nov18.1}--\eqref{nov18.3},    we
				get
				\begin{align*}
				\E\left[H(\tilde \Gamma_n)\right]
				&= O(n^{-(1+u)}) + \E\left[H(U)\right]; \\
				\sum_{(n-1)^4<j\le n^4} \E[H(U_j) ]
				&=  O(n^{-(1+u)}) + \sum_{(n-1)^4<j\le n^4}\E\left[H(\eta^j)
				\right].
				\end{align*}
				
				{Let $\tilde S$ be a simple random walk  (independent of $U$)  starting
					at the origin and let $J_n$ be the number of integers
					$j$ with $(n-1)^4 < j \leq n^4$ and such that
					$\tilde S[0,\infty) \cap U_j \neq \eset$.
					Let $\tilde{\P}$ and $\tilde \E$ are the probability and 
					expectation over $\tilde S$
					with $U$ fixed. Since the  $U_j$ are disjoint, \eqref{nov18.4}
					and the strong Markov property imply for $k \geq 1$,
					\[   \tilde{\Prob}\{J_n \geq k + 1 \mid J_n \geq k\} \leq n^{-u}.\]
					{Therefore, $\tilde{\E} [J_n] \le \tilde{\P}[J_n\ge 1]\left[1+O(n^{-u})\right]$.
						
						Since $ \tilde{\E}[J_n]=\sum_{(n-1)^4 < j \leq n^4} H(U_j)$ and  $H(U)=\tilde{\P}[J_n\ge 1]$, 
						we have 
						\[  H(U)   \geq      \left[1 - O(n^{-u})\right]   \sum_{(n-1)^4 < j \leq n^4} H(U_j).
						\]}
					Taking expectations over $U$  and using $\E[\tilde \Gamma_n]
					\leq  O(n^{-1})$, we get
					\begin{align*}
					\Big |   	\E[H(U)] - \sum_{(n-1)^4<j\le n^4}\E[H(U_j) ]   \Big | \le O(n^{-1-u}).
					\end{align*} 
					Therefore it remains to find the sets $U$ and $U_j$'s satisfying \eqref{nov18.1}-\eqref{nov18.4}.}

				Let $\sigma_j^{\pm} = \sigma_j \pm \lfloor j^{-1/4} e^{2j}\rfloor$ as in Lemma~\ref{nov18.lemma1}
				and  $\tilde \omega_j = S[\sigma_{j-1}^+,\sigma_j^-]$.
				We will let $U$ be defined as in \eqref{eq:U}, where  $U_j=\eta^j \cap \tilde \omega_j$
				unless one of the following six events occurs in which case
				we set $U_j = \eset.$  (We assume $(n-1)^4 <j \leq n^4$.)
				
				\begin{enumerate}
				\item If $j \leq (n-1)^4 + 8n$ or $j \geq n^4 - 8n$.		
				\item  If $H(\omega_j) \geq j^{-1}\, \log^2j $. 
				\item If $\omega_j \cap  C_{j -  8 \log n} \neq \eset$.
				\item  If $H(\omega_j \setminus \tilde \omega_j)\geq j^{-1-u} $.
				\item   If  it is not true that there  exist loop-free points in both $[\sigma_{j-1},\sigma_{j-1}^+]$ and $[\sigma_j^-,\sigma_j]$.
				\item    If 
				{$\sup_{x \in \tilde \omega_j} H(x,S[0, \sigma_{n^4}]  \setminus \omega_j ) \geq j^{-1/6}.$}
				\end{enumerate}
				{	We need to show that  \eqref{nov18.1}--\eqref{nov18.4} hold  for some $u>0$.}
				
				Throughout this proof we assume $n$ is large enough.	
				The definition of $U_j$ immediately implies
				\eqref{nov18.2}.
				Combining conditions 1 and 3, we see that  $U_j \subset A((n-1)^4
				+ 6n, n^4 - 6n)$.  Moreover, if there exists loop-free points
				in  $[\sigma_{j-1},\sigma_{j-1}^+]$ and $[\sigma_j^-,\sigma_j]$,
				then $\tilde \eta_n \cap \tilde \omega_j = \eta_j \cap 
				\tilde \omega_j$.  Therefore, the $U_j$ are disjoint and ~\eqref{nov18.1} holds.
				{Also, condition 6 immediately yields that \eqref{nov18.4} holds for $u\le 1/6$.}
				
				In order to establish \eqref{nov18.3} we first
				note that 
				\[        (\tilde \Gamma_n \cup  \eta^{(n-1)^4 + 1}
				\cup \cdots \cup \eta ^{n^4}
				) \setminus U
				\subset \bigcup_{(n-1)^4 <j\le n^4} V_j, \]
				where
				\[     V_j  = \left\{\begin{array}{ll} \omega_j & \mbox{ if } U_j =
				\eset \\
				\omega_j \setminus \tilde \omega_j & \mbox{ if } U_j = 
				\eta^j \cap \tilde \omega_j \end{array} . \right. \]
				Hence, it suffices to find $0<u\le 1/3$ such that
				\[   \sum_{(n-1)^4 < j \le n^4  } \left( \E\left[H(\omega_j);U_j = \eset\right] + 
				\E\left[H(\omega_j \setminus \tilde{\omega}_j)
				\right]\right) \leq c \, n^{-1-u}.\]
				To estimate $ \E\left[H(\omega_j \setminus \tilde{\omega}_j)
				\right]$, we use  \eqref{jun9.1}, and Lemma  \ref{jun7.lemma}  to see that except for an event of fast decaying
				probability
				\begin{equation}  \label{dec1.1}
				H(\omega_j \setminus \tilde \omega_j) \leq O(j^{-11/
					{10}}),
				\end{equation}
				and hence $\E[H(\omega_j \setminus \tilde \omega_j)]  \leq O(j^{-11/
					{10}})
				$ and 
				\[ \sum_{(n-1)^4 < j \le  n^4  }  \E\left[H(\omega_j \setminus \tilde{\omega}_j)
				\right] \leq c \sum_{(n-1)^4 < j \le  n^4  } j^{-11/10}
				\leq  c \, n^{-\frac 75}.\]
				
				For $i=1,2,3,4,5,6$, let $E_j^i$ be the event that the $i$th condition in the definition of $U_j$ holds
				but none of the previous ones hold. Since \(\{U_j = \eset\} = E_j^1 \cup \cdots \cup E_j^6\), to estimate $\E\left[H(\omega_j);U_j = \eset\right]$, we just need to estimate 
				$ \sum_{(n-1)^4 < j \le n^4  }\E\left[H(\omega_j);E_j^i \right]$ case by case. 
				\begin{enumerate}

				\item  Since $\E[H(\omega_j)]  \asymp j^{-1}$ for each $j$, we have 
				\[  \sum_{(n-1)^4 < j \le n^4  }
				\; \E[H(\omega_j);E_j^1] = O(n^{-3}). \]
				
				\item By Proposition~\ref{nov15.prop1}, $\Prob\left\{H(\omega_j) \geq j^{-1}\log^2 j \right\}$
				is fast decaying in $j$. This takes care of $ \sum_{(n-1)^4 < j \le n^4  }\E[H(\omega_j);E_j^2]$.
				
				\end{enumerate}
				
				On the event $E_j^3 \cup \cdots \cup E_j^6$, 
				we have $H(\omega_j) < j^{-1} \, \log^2 j . $
				Hence,
				\[  \E\left[H(\omega_j) \, ;\, {E_j^3 \cup \cdots \cup E_j^6} \right]
				\leq \frac{\log^2j}{j } \, \Prob(E_j^3 \cup \cdots \cup E_j^6).\]
				In particular, 
				it suffices 
				to prove that  there exists $u>0$ such that 
				\[
				\textrm{ 	$\Prob(E_j^i) \leq j^{-u}$	for $i=3,4,5,6$ and   $(n-1)^4+8n<  j < n^4-8n$.}
				\]
				\begin{enumerate}		\setcounter{enumi}{2}
				\item  \eqref{nov15.1} in Lemma~\ref{standardlemma} gives \(（\P(E^3_j)\le \Prob\{\omega_j \cap C_{j- \log j} \neq \eset\}\leq O(j^{-2}).\)
				\item  The bound on $\P(E^4_j)$ is already done in \eqref{dec1.1}.
				\item  Let $I = I_{\delta,n}$ be as in  \eqref{jun7.4}
				substituting in  $n = e^{2j} \, j^{1/16}$ and $\delta =
				j^{-7/16}$ so that $\delta n = e^{2j} \, j^{-3/8}$.  
				Using \eqref{jun7.4},
				we have  $\Prob(I) = o(j^{-1/4})$.   Note that the event that
				there is no loop-free time in either $[\sigma_{j-1},\sigma_{j-1}^+]$ or   $[\sigma_j^-,\sigma_j]$ is contained
				in the union of $I$ and the two events:
				\[        \{ \sigma_{j+1} \geq e^{2j} \, j^{1/16} \}\qquad\textrm{and}\qquad   \left\{ S[e^{2j} \, j^{1/16},\infty ) 
				\cap S[0,\sigma_{j+1}] \neq \eset \right\}. \]
				The probability of the first event  is fast decaying by~\eqref{jul18.1} and the probability
				of the second is $o(1/j)$ by~\eqref{eq:srw-long-range}.   Hence
				$\P(E^5_j)=o(j^{-1/4})$.		
				\item By Lemma \ref{nov18.lemma1},  for large enough $n$, we have $\P(E^6_j)\le  j^{-1/3}$. \qedhere
				\end{enumerate}
				
				\section{Exact relation}\label{constantsec}
				In this section we first prove the elemetary lemma promised at the end of Section~\ref{subsec:outline}.
				Then we give the asymptotics of the long-range intersection probability of SRW and LERW in terms of $\hat G^2_n$ defined in Section~\ref{sec;Green}, which concludes our proof of Theorem~\ref{maintheoremjune}.
				
				\subsection{A lemma about a sequence}\label{exactsec}
				\begin{lemma}
				\label{nov21.lemma2} Suppose $\beta > 0$, $p_1,p_2,\ldots$ is a sequence of
				positive numbers with $p_{n+1}/p_n \rightarrow 1$, and 
				\begin{equation*}
				\lim_{n \rightarrow \infty} \left[\log p_n + \beta\sum_{j=1}^n p_j\right]
				\end{equation*}
				exists and is finite. Then 
				\begin{equation*}
				\lim_{n \rightarrow \infty} n \, p_n = 1/\beta .
				\end{equation*}
				\end{lemma}
				
				\begin{proof}  It suffices to prove the result for $\beta = 1$, for otherwise
				we can consider $\tilde p_n = \beta \, p_n$. 
				Let
				\[    a_n = \log p_n + \sum_{j=1}^n p_j. \]
				The hypothesis implies that
				$\{a_n\}$ is a Cauchy sequence.

				We first claim that for every $\delta > 0$, there exists $N_\delta > 0$
				such that if   $n \geq N_\delta$ and $p_n = (1+ 2\epsilon)/ n$ with 
				$\eps\ge \delta$, then there does not exist $r >n$ with 
				\begin{align*}
				p_k  &\geq \frac{1 }{k}, \quad k=n,\dots, r-1,\\
				p_{r}  & \geq \frac{1 + 3\epsilon}{r}.
				\end{align*}
				Indeed, suppose  these  inequalities hold
				for some $n,r$.  Then,
				\[   \log(p_{ r}/p_n) \geq \log \frac{1+3\epsilon}{1+2\epsilon} -
				\log(r/n)   , \]
				\[     \sum_{j=n+1}^r p_j  \geq   \log(r/n)
				-  O(n^{-1}).\]
				and hence for $n$ sufficiently large, 
				\[   a_r  - a_n   \geq  \frac 12 \,  \log \frac{1+3\epsilon}{1+2\epsilon} 
				\geq    \frac 12 \log \frac{1+3\delta}{1+2\delta } .\]
				Since $a_n$ is a Cauchy sequence, this cannot be true for large $n$.

				We next claim
				that
				for every $\delta > 0$, there exists $N_\delta > 0$
				such that if   $n \geq N_\delta$ and $p_n = (1+2\epsilon) / n$ with 
				$\eps\ge \delta$, then there  exists $r$ such that 
				\begin{equation}\label{nov16.1}
				\frac{1+\epsilon}{k} \leq   p_k   < \frac{1+3\epsilon}{k}, \quad k=n,\dots, r-1,
				\end{equation}
				\[ p_{ r}   <  \frac{1 +  \epsilon}{ r}.\nonumber
				\]
				To see this, we consider the first $r$ such that   $p_{  r}   < \frac{1 +  \epsilon}{ r}$.
				By the previous claim, if such an $r$ exists, then \eqref{nov16.1}
				holds for $n$ large enough.  If no such $r$ exists, then by the argument above
				for all $r > n$, 
				\[              a_{r} -  a_n  \geq  \log\frac{1+\epsilon}{1+2\epsilon}
				+ \frac{\epsilon}{ 2} \,  \log(r/n) - (1 + \epsilon) \, O(n^{-1}) . \]
				Since the right-hand side goes to infinity as $r\rightarrow \infty$,
				this    contradicts the fact that $a_n$ is a Cauchy sequence.

				By iterating the last assertion, we can see that for every 
				$\delta > 0$,  there exists $N_\delta > 0$
				such that if   $n \geq N_\delta$  
				and $p_n = (1+2\epsilon) / n$ with 
				$\eps\ge \delta$,  then there exists $r > n$ such that
				\[             p_r <\frac{1+2\delta}{r},\;\;\textrm{and}\;\;  p_k \leq \frac{1 + 3 \epsilon}{k}, \;\;\;\; k=
				n,\ldots,r-1.\]
				Let $s$ be the first index greater than $r$ (if it exists)
				such that either 
				\[         p_k \leq \frac{1}{k}\;\; \mbox{ or } \;\;
				p_k \geq \frac{1 + 2\delta }{k}. \]
				Using $p_{n+1}/p_n \rightarrow 1$, we can see, perhaps by choosing
				a larger $N_\delta$ if necessary, that
				\[         \frac{1 - \delta}{k} \leq p_s \leq \frac{1 + 4 \delta}{k} . \]
				If $p_s \geq (1+2\delta)/k$,  then we can  iterate this argument
				with $\epsilon \leq 2 \delta$ to see that 
				\[  \limsup_{n \rightarrow \infty} n \, p_n \leq  1+ 6 \delta . \]
				The $\liminf$ can be done similarly.
				\end{proof}

				\subsection{Long range intersection}\label{sec:long}
				
				Let $S,W$ be simple random walks with
				corresponding stopping times $\sigma_n $. We will
				assume that $S_0 = w, W_0 = 0$. Let $\eta = \LE\left(S[0,\ \sigma_n]\right)$.
				Note that we are stopping the random walk $S$ at time $\sigma_n$ but we are
				allowing the random walk $W$ to run for infinite time.
				
				\begin{proposition}
				\label{prop:3.26} There exists $\alpha < \infty$ such that if $n^{-1} \leq
				e^{-n} \, |w| \leq 1 - n^{-1}$, then 
				\begin{equation*}
				\left|\log \Prob\{W[0,\infty] \cap \eta \neq \emptyset\} - \log[ \hat
				G^2_n(w) \, \hat p_n] \right| \leq c \, \frac{\log^\alpha n}{n}.
				\end{equation*}
				In particular, 
				\begin{equation*}
				\mathbb{E}\left[H(\eta^n) \right] = \frac{8 \, \hat p_n}{\pi^2} \, \left[1 +
				O\left( \frac{\log^\alpha n}{n}\right) \right].
				\end{equation*}
				\end{proposition}
				
				Throughout this section, let  $\theta_n$ denote an error term that
				decays at least as fast as $\log^\alpha n/n$ for some $\alpha$ (with the
				implicit uniformity of the estimate over all $n^{-1} \leq e^{-n} \, |w| \leq
				1 - n^{-1}$). $\theta_n$ may vary line by line. Then
				the second assertion in Proposition~\ref{prop:3.26} follows immediately from the first and \eqref{dec16.1}.
				We can write the conclusion of the proposition as 
				\begin{equation*}
				\Prob\{W[0,\infty] \cap \eta \neq \emptyset\} = \hat G^2_n(w) \, \hat p_n\,
				\left[1 + \theta_n\right],
				\end{equation*}
				Note that $\hat G_n^2(w) \geq c/n$ if $|w|
				\leq e^{n}(1 - n^{-1})$ and hence $\hat G_n^2(w)
				\, \hat p_n \leq c/n^2$.

				We start by giving the sketch of the proof which is fairly straightforward.
				On the event $\{W[0,\infty] \cap \eta \neq \emptyset\}$ there are typically
				many points in $W[0,\infty] \cap \eta$. We focus on a particular one. This
				is analogous to the situation when one is studying the probability that a
				random walk visits a set. In the latter case, one usually focuses on the
				first or last visit. In our case with two paths, the notions of ``first''
				and ``last'' are a little ambiguous so we have to take some care. We will
				consider the first point on $\eta$ that is visited by $W$ and then focus on
				the last visit by $W$ to this first point on $\eta$.
				
				To be precise, we write 
				\begin{align*}
				\eta &= [\eta_0,\ldots,\eta_m],\\
				i &= \min\{t: \eta_t \in W[0,\infty) \},\\
				\rho &= \max\{t \geq 0 : S_{t } = \eta_i\},\\
				\lambda &= \max\{t: W_t = \eta_i \}.
				\end{align*}
				Then the event $\{\rho = j; \lambda = k; S_\rho = W_\lambda = z \}$ is the
				event that:
				
				\begin{itemize}
				\item[\textbf{I}] : \;\;\; $j < \sigma_n,\;\;\;\;S_j = z, \;\;\;\; W_k = z , $
				
				{\color{blue} }
				
				\item[\textbf{II}] :\;\;\; $\LE\left(S[0,j] \right) \cap
				\left(S[j+1,\sigma_n] \cup W[0,k] \cup W[k+1,\infty) \right) = \{z\}, $
				\item[\textbf{III}] :\;\;\; $z \not\in S[j+1,\sigma_n] \cup W[k+1,\infty).$
				\end{itemize}
				Viewing the picture at $z$, we see that \(	\Prob\left\{ \mbox{\textbf{II} and \textbf{III}}\right\} \sim \hat p_n.\)
				Using the slowly recurrent nature of the random walk paths, we expect that
				as long as $z$ is not too close to $0$,$w$, or $\partial C_n$, then \textbf{I%
				} is almost independent of (\textbf{II} and \textbf{III}).
				This then gives 
				\begin{equation*}
				\Prob \{\rho = j; \lambda = k; S_\rho = W_\lambda = z \} \sim \Prob\{S_j =
				W_k = z \} \, \hat p_n ,
				\end{equation*}
				and summing over $j,k,z$ gives 
				\begin{equation*}
				\Prob\{W[0,\infty] \cap \eta \neq \emptyset\} \sim \hat p_n \, \hat G_n^2(w).
				\end{equation*}
				The following proof makes this reasoning precise. 
				\begin{proof}[Proof of Proposition~\ref{prop:3.26}]

				Let $V$ be the event 
				\begin{equation*}
				V = \{w \not \in S[1,\sigma_n], 0 \not \in W[1,\infty)\}.
				\end{equation*}
				{Using
					$\P[0\notin W[1,\infty)]=\P^w[w\in S[1,\infty)] = G(0,0)^{-1},$
					we can see that  $|\Prob(V) - G(0,0)^{-2}| $ is  fast decaying.
					Let $\tau = \max\{j: W_j = 0\}$.  Then
					$     \Prob\{\tau > \sigma_{ \log^2 n}\} $
					and  
					$\Prob\{S[0,\infty) \cap C_{\log^2 n} \neq \eset\}$
					are fast decaying and hence so is 
					\begin{equation*}
					\left| \Prob\{W[0,\infty] \cap \eta \neq \emptyset \mid V\} - \Prob%
					\{W[0,\infty] \cap \eta \neq \emptyset \} \right|  
					\end{equation*}	}
				%
				Therefore it suffices to show that 
				\begin{equation}  \label{dec4.1}
				\Prob\left[ V \cap \{W[0,\infty] \cap \eta \neq \emptyset\}\right] = \frac{%
					\hat G^2_n(w)}{G(0,0)^2} \, \hat p_n \, \left[1 + \theta_n\right].
				\end{equation}
				
				Let $E(j,k,z),E_z$ be the events 
				\begin{equation*}
				E(j,k,z) = V \cap \{\rho = j; \lambda = k; S_\rho = W_\lambda = z \},
				\;\;\;\; E_z = \bigcup_{j,k = 0}^\infty E(j,k,z) .
				\end{equation*}
				Then \(	\Prob\left[ V \cap \{W[0,\infty] \cap \eta \neq \emptyset\}\right] = \sum_{z
					\in C_n} \Prob(E_z).\)
				Let 
				\begin{equation*}
				C_n^{\prime }=C_{n,w}^{\prime }= \{z\in C_n: |z| \geq n^{-4} e^{n}, \; |z-w|
				\geq n^{-4} \, e^{n}, \; |z| \leq (1-n^{-4}) e^{n}\}.
				\end{equation*}
				We can use the easy estimate \(\Prob(E_z) \leq G_n(w,z) \, G(0,z)\)	to see that 
				\begin{equation*}
				\sum_{z \in C_n \setminus C_n^{\prime }} \Prob(E_n) \leq O(n^{-6}) ,
				\end{equation*}
				so it suffices to estimate $\Prob(E_z)$ for $z \in C_n^{\prime }$.
				
				We will translate so that $z$ is the origin and will reverse the paths $S[0,\rho]$ and	$W[0,\lambda]$. 
				Let  $\omega^1,\ldots,%
				\omega^4$ be four independent simple random walks starting at the origin and let $x = w-z, y =-z$.
				Let $l^i $ denote the smallest index $l$ such that $|\omega^i_l -y| \geq e^{n}$. 
				Using the fact that reverse loop-erasing has the
				same distribution as forward loop-erasing, we see that $\Prob[E(j,k,z)]$ can be
				given as the probability of the following event 
				\begin{equation*}
				\left(\omega^3[1,l^3] \cup \omega^4[1,\infty)\right) \cap \\LE(\omega^1[0,j])
				= \eset,
				\end{equation*}
				\begin{equation*}
				\omega^2[0,k] \cap \LE(\omega^1[0,j]) = \{0\},
				\end{equation*}
				\begin{equation*}
				j < l^1, \;\;\;\; \omega^1(j) = x, \;\;\;\; x \not\in \omega^1[0,j-1],
				\end{equation*}
				\begin{equation*}
				\omega^2(k) = y, \;\;\;\; y \not \in \omega^2[0,k-1],
				\end{equation*}
				where we translate the time reversal of $S[0,j]$, the time reversal of $W[0,k]$, the path $W[k,\infty)$ and $S[j,\infty)$ into $\omega_1,\omega_2,\omega_3,\omega_4$ respectively.
				Note that $z\in C_n'$ implies 
				\begin{equation*}
				n^{-1} \, e^n \leq |y|, |x-y|, |x| \leq e^n \, [1 - n^{-1}],
				\end{equation*}
				
				We now rewrite this. We fix $x,y$ and let $C^y_n = y + C_n$. Let $%
				W^1,W^2,W^3,$ $W^4$ be independent random walks starting at the origin and
				let  
				\begin{align*}
				&{T^3=\infty},\qquad \textrm{and}\qquad T^i = T^i_n = \min\{j: W^i_j\not\in C^y_n\}\textrm{ for }i=1\textrm{ and }4;\\
				&\tau^1 = \min\{m: W^1_m = x\}, \qquad \textrm{and}\qquad  \tau^2 = \min\{m: W^2_m = y\};\\
				&	\hat \Gamma = \hat \Gamma_n = \LE \left(W^1[0,\tau^1]\right).
				\end{align*}{We also override the notation $S$ to denote an simple random walk on $\Z^4$ starting from 0}.
				{ Let $E$ be  the event 
					\begin{equation*}
					\hat \Gamma \cap \left(W^2[1,\tau^2] \cap W^3[1,T^3]\right) = \eset \qquad \textrm{and}\qquad
					\hat \Gamma \cap W^4[0,T^4] = \{0\}.
					\end{equation*}
					Then $\P\{E, \tau^1< T^1,  \tau^2< \infty\}$ equals $\P\{V\cap W[0,\infty) \cap \eta\}$ in \eqref{dec4.1}.
					Note that  
					\begin{align*}
					\Prob\{\tau^1 < T^1\} &= \frac{G_n(w,z)}{ G_n(w,w)} = \frac{G_n(w,z)}{G(0,0)}+ o(e^{-n}),\\
					\Prob\{\tau^2 < \infty \} &= \frac{G(0,y)}{ G(y,y)} =\frac{G(0,z)}{ G(0,0)}.
					\end{align*}  
					Therefore in order to prove \eqref{dec4.1}, it suffices to prove that 
					\begin{equation}\label{eq:key}
					p_n^{\prime }(x,y):=\P\{E\mid \tau^1< T^1,  \tau^2< \infty\}= \hat p_n \, \left[1 + \theta_n\right].
					\end{equation}}
				We write $Q$ for the distribution of $W_1,W_2,W_3,W_4$ under the conditioning $\{\tau^1< T^1,  \tau^2< \infty\}$  in \eqref{eq:key}. 
				Then consider two events $E_1,E_2$ as follows. Let $\hat W = W^2[1,\tau^2] \cup W^3[1,T^3] \cup
				W^4[0,T^4]$ and let $E_0,E_1,E_2$ be the events 
				\begin{equation*}
				E_0 = \{ 0 \not\in W^2[1,\tau^2] \cup W^3[1,T^3] \},
				\end{equation*}
				\begin{equation*}
				E_1 = E_0 \cap \left\{\hat W \cap \hat \Gamma \cap C_{n-\log^3n} = \{0\}
				\right\},
				\end{equation*}
				\begin{equation*}
				E_2 = E_1 \cap \left\{\hat W \cap \Theta_n = \emptyset \right\},
				\end{equation*}
				where \(		\Theta_n = W^1[0,\tau^1] \cap A(n-\log^3n,2n).\)
				Since $Q$-almost surely 
				\begin{equation*}
				\hat \Gamma \cap C_{n-\log^3n} \subset \hat \Gamma \subset \Theta_n
				\cup(\hat \Gamma \cap C_{n-\log^3n} ),
				\end{equation*}We have
				\[
				Q(E_2) \le p'_n(x,y) \le Q(E_1).
				\] Now to prove \eqref{eq:key}. it suffices to show
				\begin{equation}\label{eq:E_1}
				Q(E_1) = \hat p_n \, \left[1 +\theta_n \right],
				\end{equation}
				\begin{equation}\label{eq:E_2}
				Q ( E_1 \setminus E_2) \leq n^{-1} \, \theta_n.
				\end{equation}
				For each $z\in \Z^4$ let 
				\begin{equation*}
				\phi _{x}(z )=\phi _{x,n}(z)=\Prob^{z}\{\tau ^{1}<T_{n}^{1}\}\qquad \textrm{and}\qquad \phi _{y}(z)=\Prob^{z }\{\tau ^{2}<\infty \}.
				\end{equation*}
				Therefore for any  path $\omega
				=[0,\omega ^{1},\ldots ,\omega ^{m}]$ with $0,\omega ^{1},\ldots ,\omega
				^{m-1}\in C_{n}^{y}\setminus \{x\}$, 
				\begin{equation}
				Q\{[W_{0}^{1},\ldots ,W_{m}^{1}]=\omega \}=\Prob\{[S_{0},\ldots
				,S_{m}]=\omega \}\,\frac{\phi _{x}(\omega ^{m})}{\phi _{x}(0)},
				\label{dec4.2}
				\end{equation}
				Similarly if $\omega
				=[0,\omega ^{1},\ldots ,\omega ^{m}]$ is a path with $y\not\in \{0,\omega
				^{1},\ldots ,\omega ^{m-1}\}$,  then
				\begin{equation*}
				Q\{[W_{0}^{2},\ldots ,W_{m}^{2}]=\omega \}=\Prob\{[S_{0},\ldots
				,S_{m}]=\omega \}\,\frac{\phi _{y}(\omega ^{m})}{\phi _{y}(0)}.
				\end{equation*}%
				Let $\zeta \in \{x,y\}$.  By \eqref{harmonic}, there exists a fast decaying sequence $\eps_n$ such that 
				\begin{equation}\label{eq:RN}
				\phi _{\zeta }(z)=\phi _{\zeta }(0)\,\left[ 1+O(\epsilon _{n})\right] \qquad \textrm{if }|z|\leq e^{n}\,e^{-\log ^{2}n}.
				\end{equation}%
				This implies that $W^i,S$ ($i=1,2$) can be coupled on the same probability space
				such that, except on an event of probability $O(\epsilon_n)$, 
				\begin{equation*}
				\hat \Gamma \cap C_{n-\log^3n} = \LE\left(S[0,\infty)\right) \cap
				C_{n-\log^3n} .
				\end{equation*}
				Therefore $Q(E_1)-p_{n-\log n^3n}$ is fast decaying and hence \eqref{eq:E_1}   follows from Proposition~\ref{nov22.cor5}.

				To prove \eqref{eq:E_2}, consider the following events whose union covers $E_1\setminus E_2$:
				\begin{align*}
				&F^2=E_1\cap \{W^2[1,\tau^2]\cap \Theta_n\neq \emptyset\},\qquad F^3=E_1\cap \{W^3[1,T^3]\cap \Theta_n\neq \emptyset\},\\
				&F^4=E_1\cap \{W^4[1,T^4]\cap \Theta_n\neq \emptyset\}.
				\end{align*}
				We are now going to prove $Q(F^i)\le n^{-1}\theta_n$ for $i=2,3,4$, thus proving \eqref{eq:E_2}.
				
				{
					By  Proposition~\ref{nov15.prop1}, we can find a fast decaying sequence $\delta_n$ such that 
					\begin{equation*}
					\Prob\left\{H\left(S[0,\infty) \cap A(n-1,n) \right) \geq \frac{\log^2 n}{n}
					\right\} \leq \delta^{100}_n.
					\end{equation*}
					Let  \(\rho = \rho_n = \min\{j: |W_j^1 - x| \leq e^n \, \delta_n\}.\) 
					Using the  strong Markov property of $W^1$ under $Q$,  we can find a constant $\alpha>0$ such that 
					\begin{equation*}
					Q\left\{W^1[\rho,\tau^1] \not \in \{|z -x| \leq e^n \, \sqrt {\delta_n}\}\right\} = O(\delta^\alpha_n).
					\end{equation*}
					Since $|x| \geq e^{-n} \, n^{-1}$, we know that 
					\begin{equation*}
					H\left (\{|z -x| \leq e^n \, \sqrt {\delta_n}\} \right) \leq n^2 \, \delta_n.
					\end{equation*}
					By the strong Markov property, for all $w\in C^y_n$ with $|w-x|\ge e^n\sqrt{\delta_n}$,
					\[
					\phi(0) /\phi(w) \gtrsim \P[\rho_n<T^1]\ge \delta^{50}_n.
					\]
					Using \eqref{dec4.2}
					we have
					\begin{equation}\label{eq:fdQ}
					Q\left\{H\left(W^1[0,\tau^1] \cap A(n-1,n) \right) \geq \frac{\log^2 n}{n}
					\right\} \leq \frac{O(\delta^{100}_n) }{\P[\rho_n<T^1]} + O(\delta^\alpha_n),
					\end{equation}which is fast decaying. Therefore, $Q[W^1[0,\tau^1]\notin \niceset_n]$ is fast decaying.
					Let
					\begin{equation}\label{eq:Q}
					R_n= H\left(W^1[0,\tau^1] \cap A(n-2\log^3 n, n+1) \right).
					\end{equation}
					Let $\sigma=\inf\{m: W^3_m\notin C_{n-\log^3n} \}$ and  $E=\{W^3[1,\sigma]\cap \hat \Gamma= \emptyset\}$.
					Let $\bar Q$ be the probability measure conditioning on $W^1,W^2,W^4$. Then 	
					\[
					\bar Q[F^3] \le \bar Q(E) \sum_{z\in \partial C_{n-\log^3n}}\bar Q[S_{\sigma}=z\mid E]  Q[W^{3}[\sigma,T^3]\cap \hat{\Gamma}=\eset \mid W^3_\sigma=z].
					\]			On the event $\hat \Gamma\in \nicesettwo_{n-\log^3n}$, by Proposition~\ref{cond}, there exists $c<\infty$ such that
					\begin{align*}
					&\sum_{z\in \partial C_{n-\log^3n}}\bar Q[W^3_{\sigma}=z\mid E]  \bar Q[W^{3}[\sigma,T^3]\cap \hat{\Gamma}\neq \eset \mid W^3_\sigma=z]\\
					&\le c\bar Q[W^{3}[\sigma,T^3]\cap \hat{\Gamma}\neq \eset]\\
					&\le c\left( H(\hat \Gamma \cap A(n-2\log^3 n, n+1))+ Q[W^3[\sigma,\infty)\cap C_{n-2\log^3n}\neq\emptyset]\right)\\
					&\le  c(R_n+Q[W^3[\sigma,\infty)\cap C_{n-2\log^3n}\neq\emptyset])\le cR_n+ O(e^{-\log^3n}) .
					\end{align*}	 
					Note that  $\Es(\hat \Gamma)\le 2^{-(n-\log^3n)/4}$ when $\hat \Gamma\in \nicesettwo_n$. Moreover, $Q[\hat \Gamma_n\notin \niceset_n]$ is fast decaying.
					Applying the union bound to \eqref{eq:fdQ}, we have $Q[R_n\geq \log^7n/n]$ is fast decaying. 
					Averaging over $W^1,W^2,W^4$, we have 
					\[
					Q(F^3)\le O(\log^7n/n)Q(E_1)+\eps_n
					\]
					where $\eps_n$ is fast decaying.
					This gives $Q(F^3)\le n^{-1}\theta_n$ for some $\theta_n$. The same argument shows $Q(F^4)\le n^{-1}\theta_n$

					We still need a similar result to conclude $Q(F^2)\le n^{-1}\theta_n$.  
					By \eqref{eq:RN}, we can couple an usual simple random walk $S$ with $W^2$ (under the $Q$-probability) 
					such that $S$ and $W^2$ agree until
					$\inf\{m: S_m\notin C_{n-\log^3n} \}$ except for an event of fast decaying probability. 
					The same argument as above reduces proving 
					$Q(F^2)\le n^{-1}\theta_n$ to showing that there exists $\theta_n$ such that except for an event of fast decaying $Q$-probability, 
					\begin{equation*}
					Q\{W^{2}[0,\tau ^{2}]\cap (\hat{\Gamma}\cap A(n-2\log ^{3}n,n+1))\neq \eset%
					\mid \hat{\Gamma}\}\leq \theta _{n}.
					\end{equation*}
					which follows from a similar argument for the bound for $Q[R_n\ge \log^7n/n]$ above.
				}\qedhere
				\end{proof}
				As explained in Section~\ref{subsec:outline}, combined with Lemma~\ref{nov21.lemma2}, we conclude the proof of Theorem~\ref{maintheoremjune}. 
				Inserting $\hat p_n\sim \pi^2/24$ back to Proposition~\ref{prop:3.26},  we get
				\begin{corollary}
				If $n^{-1} \leq e^{-n} \, |w| \leq 1 - n^{-1} $, then 
				\begin{equation*}
				\Prob\{W[0,\infty) \cap \eta \neq \eset\} \sim \frac{\pi^2 \, \hat G_n^2 (w)%
				} {24 \, n}.
				\end{equation*}
				More precisely, 
				\begin{equation*}
				\lim_{n \rightarrow \infty}\; \max_{n^{-1} \leq e^{-n} \, |w| \leq 1 -
					n^{-1} } \; \left|24 \, n\, \Prob\{W[0,\infty) \cap \eta \neq \eset\} -
				\pi^2 \, \hat G^2_n(w) \right| = 0 .
				\end{equation*}
				\end{corollary}
				By a very similar proof one can show the following variant of Proposition~\ref{prop:3.26} that implies Theorem~\ref{thm:long-range}: 
				\begin{proposition}\label{prop:2pt}
				In the setting of Theorem~\ref{thm:long-range}, there exists $\alpha < \infty$ such that if $n^{-1} \leq
				e^{-n} \, |w| \leq 1 - n^{-1}$, then 
				\begin{equation*}
				\left|\log \Prob\{W[0, \sigma^W_n] \cap \LE(S[0,\sigma_n]) \neq \emptyset\} - \log[
				G^2_n(w) \, \hat p_n] \right| \leq c \, \frac{\log^\alpha n}{n}.
				\end{equation*}
				
				\end{proposition}

				\section{Two-sided loop-erased random walk\label{2side}}  \label{sec:two-sided}
				We start by finishing the proof of Theorem~\ref{thm:inter}.  
				{ Let \(\tilde Z_n = \Es(\Gamma;n).\)   We first note that the limit
					\(\lim_{ n\rightarrow \infty } n^{1/3}
					\, \tilde Z_n\) 
					exists almost surely.  We only need
					condsider  $\Gamma\in \nicesettwo$ since otherwise the limit is 0 almost surely.
					In this case, existence
					is established  by Corollary	\ref{cor.jun22}, Proposition~\ref{nov22.prop3} and the fact that $3nh_n\sim 1$.   
					Recall \eqref{jul19.4}. We have 
					\[   \E\left[\tilde Z_n^r \right]
					= \E \left[  Z_n^r\right]
					\left[1 + O(\log^4 n/n) \right].
					\]
					Since
					for each $r$, $\E\left[n^{r/3} \ \tilde Z_n^r\right]$
					is uniformly bounded, we also get the limit in
					$L^p$ for each $p>0$.  We then get
					Theorem \ref{thm:inter} using \eqref{jul18.1}.}

				If $\eta$ is an infinite  (one-sided)
				path starting at the origin,
				and $W$ is a simple random walk with stopping
				times $\sigma^W_n=\inf\{W_j\notin C_n\}$, we define 
				\[       \phi_\eta(x)  =
				\lim_{n \rightarrow \infty}
				\, n^{1/3 }
				\, \Prob^x\{W[0,\sigma^W_n] \cap \eta = \eset\}.\]
				\[  \nabla \phi_\eta =  \frac 18 \sum_{|y| = 1}
				\phi_\eta(y)  .\]
				We let $\saws$ denote the set 
				infinite self-avoiding
				paths starting at the origin
				such that the limit above exists and
				is finite for all $x \in \Z^4$ and  let $\saws^+$
				be the set of such $\eta$ with
				$\nabla \phi_\eta > 0$. 
				We can restate Theorem~\ref{thm:inter}.   as follows:  if
				$\eta =   \LE(S[0,\infty))$, then with probability
				one, $\eta \in \saws$ 
				and with
				positive probability $\eta \in \saws^+$.
				Moreover, $ \nabla \phi_\eta$ is in $L^p$ for all $p$.
				
				We can now construct the two-sided loop-erased
				random walk in $\Z^4$.  This is a measure on 
				doubly-infinite self-avoiding paths
				\[   \omega = \left[\cdots  \omega_{-2},  \omega_{-1},
				\  \omega_0, \  \omega_1,  \omega_2. \cdots\right] \]
				with $  \omega_0=0$.  Each $\omega$ can be described in terms of the
				sequence of one-sided paths
				\[     \eta^j = \left[\  \omega_{j}   ,
				\  \omega_{j-1} ,   \omega_{j-2} ,
				\ldots \right] - \omega_j, \;\;\;\;
				j\in \N, \]
				where $\eta^{j+1}$ is obtained from $\eta^j$ by 
				choosing
				$|x| = 1$, attaching $x$ to
				the beginning of $\eta^j$, and translating by $-x$.
				The transition probabilities for the chain are specified
				by saying that if $\eta^j = \eta$,
				then $x$ is chosen with probability
				$\phi_\eta(x)/
				\nabla \phi_\eta$.  
				We choose $\eta^0 =  \eta = \LE(S[0,\infty))$
				tilting by $\nabla \phi_\eta/\E[\nabla \phi_\eta]$.

				We give another definition.  Suppose
				\[   \eta = [\eta_{-k},\eta_{-k-1},\ldots, \eta_{j-1},
				\eta_j] \]
				is a (finite) self-avoiding walk in $\Z^4$.  Let
				$S,W$ be independent random walk starting at
				$z= \eta_{-k}, w = \eta_j$, respectively,  with
				corresponding stopping times $\sigma^S_n,
				\sigma^W_n$, respectively, and
				let $E_n$ be the event
				\[      S[1, \sigma^S_n] \cap \eta = \eset,\;\;\;\;
				W[1,\sigma^W_n] \cap \eta = \eset , \]
				\begin{equation}  \label{jul21.1}
				\LE(S[1,\sigma^S_n] ) \cap W[1,
				\sigma^W_n] = \eset . 
				\end{equation}
				We define
				\begin{equation}  \label{jul21.2}
				\Es(\eta) = \lim_{n \rightarrow \infty} n^{1/3}\,
				\Prob^{z,w}[E_n].
				\end{equation}
				It follows from our theorems that the
				limit exists.  Moreover (see \cite[Chapter~9]{lawler-limic}), the limit would
				be the same if condition \eqref{jul21.1}
				in the event $E_n$ is replaced by
				\[       S[1,\sigma^S_n]  \cap \LE(W[1,
				\sigma^W_n] ) = \eset . \]
				From this we see that $\Es(\eta)$ is translation invariant
				and also invariant under path reversal.  The
				probability that the two-sided loop-erased walk
				produces $\eta$ is
				\[       8^{-(j+k)} \, F_\eta \, \frac{\Es(\eta)}{\Es(0)}.\]
				Here $\Es(0)$ is the quantity where $\eta$
				is the trivial walk $[0]$ and $F_\eta$ is a loop
				term that can be given in several ways, e.g.,
				\[   F_\eta = \prod_{i=-k}^j G_{A_i}(\eta_i,\eta_i),\]
				where $A_{i} = \Z^4 \setminus \{\eta_{-k},\ldots,
				\eta_{i-1} \}$.  Although not immediately obvious, this
				quantity depends only on the set $\{\eta_{-k},
				\ldots,\eta_j\}$ and not on the ordering of the points.

				\begin{itemize}
				\item  For $d\geq 5$, it was constucted 
				in \cite{CCLERW}, where one can define $\phi_\eta$
				by
				\[   \Prob\{W[0,\infty) \cap \eta = \eset\}.\]
				In this case, the marginal distribution on
				the past or future of the path is absolutely continuous
				with respect to the one-sided measure with a bounded
				Radon-Nikodym derivative.
				\item  For $d=4$, it is absolutely continuous with
				an $L^p$, but not uniformly bounded, derivative.
				\item  In \cite{lawler-twoside}
				the two-sided walk is constructed for $d=2,3$
				using \eqref{jul21.2} replacing $n^{1/3}$ with a sequence $a_n$.
				(The proof there also works for $d > 3$ but does not
				give as strong a result as above.)
				If $d=2$, it is known that $a_n ={e^{3n/4}}$ works; for
				$d=3,$ it is expected that we can choose $a_n
				= {e^{\beta n}}$ for an appropriate $\beta$ but this has
				not been proven.

				\end{itemize}
				
				\section{Gaussian limits for the spin field} \label{sec:field}
				In this section, we start by reviewing some known facts of UST and random
				walk Green's function, then proving Theorem \ref{BGF} by applying main
				estimates of LERW. 
				
				{}
				
				\subsection{Uniform spanning trees}
				
				\label{sec:USF}
				
				Here we review some facts about the uniform spanning forest (that is, wired
				spanning trees) on finite subsets of $\mathbb{Z}^{d}$  on $\mathbb{Z}^{d}$. Most of the facts extend to general graphs
				as well. For more details, see \cite[Chapter 9]{lawler-limic}.
				
				Given a finite subset $A\subset \mathbb{Z}^{d}$, the uniform wired spanning
				tree in $A$ is a subgraph of the graph $A \cup \{\partial A\}$, choosing
				uniformly random among all spanning trees of $A \cup \{\partial A\}$. (A
				spanning tree $\mathcal{T}$ is a subgraph such that any two vertices in $%
				\mathcal{T}$ are connected by a unique simple path in $\mathcal{T}$). We
				define the \textit{uniform spanning forest (USF)} on $A$ to be the uniform
				wired spanning tree restricted to the edges in $A$. One can also consider
				the uniform spanning forest on all of $\mathbb{Z}^d$ \cite{Pem91,BLPS}, but
				we will not need this construction.
				
				The uniform wired spanning tree, and hence the USF, on $A$ can be generated
				by Wilson's algorithm \cite{Wil96} which we recall here:
				
				\begin{itemize}
				\item Order the elements of $A = \{x_1,\dots,x_k\}$.
				
				\item Start an SRW at $x_1$ and stop it when in reaches $%
				\partial A$ giving a nearest neighbor path $\omega$. Erase the loops 
				chronologically to produce  $\eta =  \LE(\omega)$. Add all
				the edges of $\eta$ to the tree which now gives a tree $\mathcal{T}_1$ on a
				subset of $A \cup \{\partial A\}$ that includes $\partial A$.
				
				\item Choose the vertex of smallest index that has not been included and run
				a simple random walk until it reaches a vertex in $\mathcal{T}_1$. Erase the
				loops and add the new edges to $\mathcal{T}_1$ in order to produce a tree $%
				\mathcal{T}_2$.
				
				\item Continue until all vertices are included in the tree.
				\end{itemize}
				
				Wilson's theorem states that the distribution of the tree is independent of
				the order in which the vertices were chosen and is uniform among all
				spanning trees. In particular we get the following.
				
				\begin{itemize}
				\item If $x,y \in A$, let $S^x,S^y$ be two independent SRW's starting from $%
				x,y$ respectively. Then the probability  that $x,y$ are in the same component
				of the USF equals to  $\P \{\LE(\omega^x) \cap \omega^y = \emptyset\}.$
				
				\end{itemize}
				
				Using this characterization, we can see the three regimes for the dimension $%
				d$. Let us first consider the probabilities that neighboring points are in
				the same component. Let $q_N$ be the probability that a nearest neighbor of
				the origin is in a different component as the origin when $A=A_N$. Then 
				\begin{equation*}
				q_\infty := \lim_{N \rightarrow \infty } q_N >0 , \;\;\;\;\; d \geq 5,
				\end{equation*}
				\begin{equation*}
				q_N \asymp (\log N)^{-1/3}, \;\;\;\; d = 4 ,
				\end{equation*}
				For $d < 4$, $q_N$ decays like a power of $N$. For far away points, we have
				
				\begin{itemize}
				\item If $d > 4$, and $|x| = n$, the probability that $0$ and $x$ are in the
				same component is comparable to $|x|^{4-d}$. This is true even if $N =\infty$%
				.
				
				\item If $d = 4$ and $|x| = n$, the probability that that $0$ and $x$ are in
				the same component is comparable to $1/\log n$. However, if we chose $N =
				\infty$, the probability would equal to one.
				\end{itemize}
				
				The last fact can be used to show that the USF in all of $%
				\mathbb{Z}^4$ is, in fact, a tree. For $d < 4$, the probability that $0$ and 
				$x$ are in the same component is asymptotic to $1$ and our construction is
				not interesting. This is why we restrict to $d \geq 4$.

				\subsection{Proof of Theorem \protect\ref{BGF}}
				
				\label{subsec:moment} Here we give the proof of the theorem by applying Theorem~\ref%
				{maintheoremjune} and Proposition~%
				\ref{prop:2pt}. We will only consider the $d=4$ case here. It suffices to prove the result for $m =1, h_1 = h$, as the general result
				follows by applying the $k=1$ result to any linear combination of $h_1,\ldots,h_m$.
				
				We fix $h\in C_{0}^{\infty }$ with $\int h=0$ and allow implicit constants
				to depend on $h$. We will write just $Y_{x}$ for $Y_{x,n}$. Let $K$ be such
				that $h(x)=0$ for $|x|>K$ .
				
				Let us write $L_{n}=n^{-1}\,\mathbb{Z}^{4}\cap \{|x|\leq K\}$ and  $a_{n}=\sqrt{3\log n}$. Let
				\begin{equation*}
				\langle h,\phi _{n}\rangle =n^{-2}\,a_{n}\sum_{x\in
					L_{n}}h(x)\,Y_{nx}=n^{-2}\,a_{n}\sum_{nx\in A_{nK}}h(x)\,Y_{nx},
				\end{equation*}%
				Let $q_N(x,y)$ be the probability that $x,y$ are in the same component of
				the USF on $A_N$. Note that 
				\begin{equation*}
				\mathbb{E}\left[Y_{x,n} \, Y_{y,n}\right] = q_N(x,y),
				\end{equation*}
				\begin{equation*}
				\mathbb{E}\left[ \langle h, \phi_n \rangle ^2 \right] = n^{-4} \sum_{x \in
					L_n} h(x) \,h(y) \, a_n^2 \, q_N(nx,ny).
				\end{equation*}
				To estimate $\mathbb{E}\left[ \langle h, \phi_n \rangle ^2 \right]$, we need the follow two lemmas.
				
				{ 
					Let $G_N(x,y)$ denote the usual random walk Green's function on $A_N$, and
					\begin{equation*}
					G_N^2(x,y) = \sum_{z \in A_N} G_N(x,z) \, G_N(z,y),
					\end{equation*}
					Note that here the meaning of $G_N$ is not the same as $G_n$ in Section~\ref{sec;Green} with $n=N$ as $A_N\neq C_N$. 
					\begin{lemma}
					\label{feb3.lemma1}There exists $c_0 \in \mathbb{R}$ such that
					if $|x|,|y|, |x-y| \leq N/2,$ then 
					\begin{equation*}
					G_N^2(x,y) = \frac{8}{\pi^2} \, \log\left[\frac N {|x-y|}\right] + c_0 +
					O\left(\frac{|x| + |y| + 1}{N} + \frac{1}{ |x-y|}\right) .
					\end{equation*}
					\end{lemma}	
					\begin{proof}
					Let $\delta = N^{-1} [1+|x| + |y|]$ and
					note that $|x-y| \leq \delta \, N$.  Since
					\begin{eqnarray*}
					\sum_{|w| < N(1-\delta)} G_N(x-y,w) \, G_N(0,w)
					& \leq  & 
					\sum_{w \in A_N} G_N(x,w) \, G_N(y,w)\\
					& 
					\leq & \sum_{|w| \leq N(1 + \delta)} G_N(x-y,w) \, G_N(0,w), 
					\end{eqnarray*}
					Lemma~\ref{feb3.lemma1} is reduced to the case $y=0$, which  then follows from Lemma~\ref{lem:Green2} and~\ref{lem:Green}.
					\end{proof}
					\begin{lemma}
					\label{lem:2.4} \label{important} There exists a sequence $r_{n}$ with $%
					r_{n}\leq O(\log \log n)$, a sequence $\epsilon _{n}\downarrow 0$, such that
					if $x,y\in L_{n}$ with $|x-y|\geq 1/\sqrt{n}\,$, 
					\begin{equation*}
					\left\vert a_{n}^{2}\,q_{N}(nx,ny)-r_{n}+\log {|x-y|}\right\vert \leq
					\epsilon _{n}.
					\end{equation*}
					\end{lemma}
					\begin{proof} 
					In light of Wilson's algorithm,  Lemma~\ref{lem:Green} and Proposition~\ref{prop:2pt}  yield Lemma \ref{important} in the $y=0$ case.
					The general case can be reduced to the $y=0$ case as in Lemma~\ref{feb3.lemma1} by recentering.
					\end{proof}	
				}

				An upper bound for $q_{N}(x,y)$ can be given in terms of the probability
				that the paths of two independent random walks starting at $x,y$, stopped
				when they leave $A_N$, intersect. This gives 
				\begin{equation*}
				q_{N}(x,y)\leq \frac{c\,\log [N/|x-y|]}{\log N}.
				\end{equation*}%
				Let $\delta_n = \exp\{- (\log \log n)^2\}$, which is a
				function that decays faster than any power of $\log n$.  Then 
				\begin{equation}
				q_{N}(x,y)\leq c\,\frac{(\log \log n)^{2}}{\log n},\;\;\;\;\;|x-y|\geq
				n\,\delta _{n}.  \label{feb3.1}
				\end{equation}

				%
				

				It follows from Lemma~\ref{lem:2.4}, \eqref{feb3.1}, and the trivial
				inequality $q_N \leq 1$, that 
				\begin{eqnarray*}
				\mathbb{E}\left[ \langle h, \phi_n \rangle^2 \right] & = & o(1) + n^{-4} \,
				\sum_{x,y \in L_n} \left[r_n - \log|x-y|\right] \, h(x) \, h(y) \\
				& = & o(1) - n^{-4} \,\sum_{x,y \in L_n} \log {|x-y|} \, h(x) \, h(y) \\
				& = & o(1) - \int h(x) \, h(y) \, \log |x-y|\, dx\, dy,
				\end{eqnarray*}
				which shows that the second moment has the correct limit. The second
				equality uses $\int h = 0$ to conclude that 
				\begin{equation*}
				\frac{r_n}{n^4} \, \sum_ {x,y \in L_n} h(x) \, h(y) = o(1).
				\end{equation*}
				
				We now consider the higher moments. It is immediate from the construction
				that the odd moments of $\langle h,\phi _{n}\rangle $ are identically zero,
				so it suffices to consider the even moments $\mathbb{E}[\langle h,\phi
				_{n}\rangle ^{2k}]$. We fix $k\geq 1$ and
				allow implicit constants to depend on $k$ as well. Let $L_{\ast }=L_{n,\ast
				}^{2k}$ be the set of $\bar{x}=(x_{1},\ldots ,x_{2k})\in L_{n}^{2k}$ such
				that $|x_{j}|\leq K$ for all $j$ and $|x_{i}-x_{j}|\geq \delta _{n}$ for
				each $i\neq j$. We write $h(\bar{x})=h(x_{1})\,\dots \,h(x_{2k}).$
				
				Note that $\#L_\ast\asymp n^{8k} $ and $\#(L_n^{2k} \setminus L_{*}) \asymp
				k^2 \, n^{8k}\, \delta_n .$ In particular, 
				\begin{equation*}
				n^{-8k} \, a_n^{2k} \sum_{\bar x\in L_{n,*}^{2k}} h(x_1) \, h(x_2) \, \cdots
				\, h(x_{2k}) \, = o_{2k} (\sqrt {\delta_n}).
				\end{equation*}
				
				Then we see that 
				\begin{eqnarray*}
				\mathbb{E}\left[ \langle h,\phi _{n}\rangle ^{2k}\right] &=&n^{-8k}%
				\,a_{n}^{2k}\sum_{\bar{x}\in L_{n}^{2k}}h(\bar{x})\,\mathbb{E}\left[
				Y_{nx_{1}}\cdots Y_{nx_{2k}}\right] \\
				&=&O(\sqrt{\delta _{n}})+n^{-8k}\,a_{n}^{2k}\sum_{\bar{x}\in L_{\ast }}h(%
				\bar{x})\,\mathbb{E}\left[ Y_{nx_{1}}\cdots Y_{nx_{2k}}\right].
				\end{eqnarray*}
				
				\begin{lemma}
				\label{lem:2.5} \label{jul24.lemma1} For each $k$, there exists $c<\infty $
				such that the following holds. Suppose $\bar{x}\in L_{n,\ast }^{2k}$ and let 
				$\omega ^{1},\ldots ,\omega ^{2k}$ be independent simple random walks
				started at $nx_{1},\ldots ,nx_{2k}$ stopped when they reach $\partial A_N$%
				. Let $N$ denote the number of integers $j\in \{2,3,\ldots ,2k\}$ such that 
				\begin{equation*}
				\omega ^{j}\cap (\omega ^{1}\cup \cdots \cup \omega ^{j-1})\neq \emptyset .
				\end{equation*}%
				Then, 
				\begin{equation*}
				\P \{N\geq k+1\}\leq c\,\left[ \frac{(\log \log n)^{3}}{\log n}\right]
				^{k+1}.
				\end{equation*}
				\end{lemma}
				
				Conditioned on Lemma~\ref{lem:2.5}, we now prove Theorem~\ref{BGF} by
				verifying Wick's formula. We write $y_{j}=nx_{j}$ and write $Y_{j}$ for $%
				Y_{y_{j}}$. To calculate $\mathbb{E}[Y_{1}\cdots Y_{2k}]$ we first sample
				our USF which gives a random partition ${\mathcal{P}}$ of $\{y_{1},\ldots
				,y_{2k}\}$. Note that $\mathbb{E}[Y_{1}\cdots Y_{2k}\mid {\mathcal{P}}]$
				equals $1$ if it is an \textquotedblleft even\textquotedblright\ partition
				in the sense that each set has an even number of elements. Otherwise, $%
				\mathbb{E}[Y_{1}\cdots Y_{2k}\mid {\mathcal{P}}]=0$. Any even partition,
				other than a partition into $k$ sets of cardinality $2$, will have $N\geq
				k+1 $. Hence 
				\begin{equation*}
				\mathbb{E}[Y_{1}\cdots Y_{2k}]=O\left( \left[ \frac{(\log \log n)^{3}}{\log n%
				}\right] ^{k+1}\right) +\sum \P \left( {\mathcal{P}}_{\bar{y}}\right) ,
				\end{equation*}%
				where the sum is over the $(2k-1)!!$ perfect matchings of $\{1,2,\ldots
				,2k\} $ and $\P \left( {\mathcal{P}}_{\bar{y}}\right) $ denotes the
				probability of getting this matching for the USF for the vertices $%
				y_{1},\ldots ,y_{2k}.$
				
				Let us consider one of these perfect matchings that for convenience we will
				assume is $y_1 \leftrightarrow y_2, y_3 \leftrightarrow y_4, \ldots,y_{2k-1}
				\leftrightarrow y_{2k}$. We claim that 
				\begin{equation*}
				\P (y_1 \leftrightarrow y_2, y_3 \leftrightarrow y_4, \ldots,y_{2k-1}
				\leftrightarrow y_{2k}) = \hspace{2in}
				\end{equation*}
				\begin{equation*}
				O\left(\left[\frac{(\log \log n)^3}{\log n} \right]^{k+1}\right) + \P (y_1
				\leftrightarrow y_2) \, \P (y_3 \leftrightarrow y_4) \cdots \P (y_{2k-1}
				\leftrightarrow y_{2k}).
				\end{equation*}
				Indeed, this is just inclusion-exclusion using our estimate on $\P \{N \geq
				k+1\}$.
				
				If we write $\epsilon_{n} = \epsilon_{n,k} = (\log \log n)^{3(k+1)}/\log n$,
				we now see from symmetry that 
				\begin{eqnarray*}
				\lefteqn{\E\left [ \langle h, \phi_n \rangle^{2k} \rangle\right]} \\
				& =& O(\epsilon_n) + n^{-8k} a_n \, (2k-1)!! \, \sum_{\bar x \in L_*} \P %
				\{nx_1 \leftrightarrow nx_2, \ldots, nx_{2k-1} \leftrightarrow nx_{2k} \} \\
				& = & O(\epsilon_n) + (2k-1)!! \left[\mathbb{E}\left(\langle h,\phi_n
				\rangle^2\right) \right]^k.
				\end{eqnarray*}
				
				\subsection{Proof of Lemma \protect\ref{jul24.lemma1}}
				
				Here we fix $k$ and let $y_{1},\ldots ,y_{2k}$ be points with $|y_{j}|\leq
				Kn $ and $|y_{i}-y_{j}|\geq n\,\delta _{n}$ where we recall $\log \delta
				_{n}=-(\log \log n)^{2}$. Let $\omega ^{1},\ldots ,\omega ^{2k}$ be
				independent simple random walks starting at $y_{j}$ stopped when they get to 
				$\partial A_N$. We let $E_{i,j}$ denote the event that $\omega ^{i}\cap
				\omega ^{j}\neq \emptyset $, and let $R_{i,j}=\P (E_{i,j}\mid \omega ^{j})$.
				
				\begin{lemma}
				There exists $c < \infty$ such that for all $i,j$, and all $n$ sufficiently
				large, 
				\begin{equation*}
				\P \left\{R_{i,j} \geq c \, \frac{(\log \log n)^3} {\log n} \right\} \leq 
				\frac{1}{(\log n)^{4k}} .
				\end{equation*}
				\end{lemma}
				
				\begin{proof}  We know that there exists $c < \infty$
				such that  if $|y-z| \geq
				n \, \delta_n^2$, then the probability that
				simple random walks starting at $y,z$ stopped
				when they reach $\partial A_N$ intersect is at most
				$c(\log \log n)^2/\log n$.
				Hence there exists $c_1$ such that  
				\begin{equation}  \label{jul24.2}
				\P\left\{R_{i,j} \leq \frac{c_1 \, (\log \log n)^2}{\log n} \right\}
				\geq \frac 12 . 
				\end{equation}
				Start a random walk at $z$ and run it until
				one of three things happens:
				\begin{itemize}
				\item  It reaches $\partial A_N$ 
				\item  It gets within distance $n \, \delta_n^2$ of $y$
				\item  The path is such that the probability that a
				simple random walk starting at $y$ intersects the path
				before reaching $\partial A_N$ is greater than $c_1
				(\log \log n)^2/\log n$.
				\end{itemize}
				If the third option occurs, then we restart the walk
				at the current site and do this operation again.  Eventually
				one of the first two options will occur.   Suppose
				it takes $r$ trials of this process until one of the
				first two events occur.  Then either $R_{i,j}
				\leq r \, c_1\,(\log \log n)^2/\log n$ or the original path starting at $z$
				gets within distance $\delta_n^2$ of $y$.  The latter
				event occurs with probability $O(\delta_n) =
				o((\log n)^{-4k})$.  Also, using \eqref{jul24.2}, we can
				see the probability that it took at least $r$ steps is
				bounded by $(1/2)^{r}$.   By choosing $r = c_2 \log 
				\log n$, we can make this probability less than
				$1/(\log n)^{4k}$. 
				\end{proof}
				
				\begin{proof}[Proof of Lemma \ref{jul24.lemma1}]
				Let $R$ be the maximum of $R_{i,j}$ over all $i\neq j$
				in $\{1,\ldots,2k\}$.  Then, at least for $n$ sufficiently
				large, 
				\[         \P\left\{ R  \geq c \, \frac{(\log \log n)^3}
				{\log n} \right\} \leq \frac{1}{(\log n)^{3k}} . \]
				Let
				\[       E_{\cdot,j} =\bigcup_{i=1}^{j-1} E_{i,j} .\]
				On the event  $R < c (\log \log n)^3/\log n$, we have
				\[   \P\left\{E_{\cdot,j} \mid \omega^1,\ldots,\omega^{j-1}
				\right\}  \leq \frac{c(j-1) \, (\log \log n)^3}{\log n }
				.\]
				If $N$ denotes the number of $j$ for which $E_{\cdot,j}$
				occurs, we see that
				\[    \P \{N \geq k+1 \} \leq c \, \left[\frac{(\log \log n)^3}
				{\log n} \right]^{k+1} . \qedhere\] 
				\end{proof}

				\section*{Acknowledgements}
				We are grateful to Scott Sheffield
				for motivating this problem and helpful discussions. We thank Yuval Peres, 
				Robin Pemantle and an anonymous referee for useful comments on an earlier draft. We thank the
				Isaac Newton Institute for Mathematical Sciences, Cambridge, for support and
				hospitality during the programme \emph{Random Geometry} where part of this
				project was undertaken. 
				G. Lawler was supported by
				NSF grant DMS-1513036.
				X. Sun was supported by Simons Foundation as a Junior Fellow at Simons Society of Fellows  and by NSF grants DMS-1811092 and by Minerva fund at Department of Mathematics at Columbia University. 
				The research of W. Wu was supported in part by NSF grant DMS-1507019.
				

				\end{document}